\documentclass[12pt]{article}
\usepackage{amsmath,amsfonts,amssymb,amsthm}
\begin{document}

\newtheorem{thm}{Theorem}[section]
\newtheorem{prop}[thm]{Proposition}
\newtheorem{coro}[thm]{Corollary}
\newtheorem{conj}[thm]{Conjecture}
\newtheorem{example}[thm]{Example}
\newtheorem{lem}[thm]{Lemma}
\newtheorem{rem}[thm]{Remark}
\newtheorem{hy}[thm]{Hypothesis}
\newtheorem*{acks}{Acknowledgements}

\theoremstyle{definition}
\newtheorem{de}[thm]{Definition}

\newcommand{\C}{{\mathbb{C}}}
\newcommand{\Z}{{\mathbb{Z}}}
\newcommand{\N}{{\mathbb{N}}}
\newcommand{\Zr}{\Z^r}
\newcommand{\Zall}{\Z\times \Z^{r}}

\newcommand{\End}{{{\rm End} }}
\newcommand{\Hom}{{\rm Hom}}
\newcommand{\Image}{{\rm Im}}
\newcommand{\Ker}{{{\rm Ker}}}
\newcommand{\Mod}{{{\rm Mod}}}
\newcommand{\Res}{{\rm Res}}

\newcommand{\g}{{\mathfrak{g}}}
\newcommand{\gc}{{\hat{\g}}}
\newcommand{\ga}{{\tilde{\g}}}
\newcommand{\borel}{{\mathfrak{b}}}
\newcommand{\Borel}{{\mathfrak{B}}}
\newcommand{\nil}{{\mathfrak{n}}}
\newcommand{\Nil}{{\mathfrak{N}}}
\newcommand{\h}{{\mathfrak{h}}}
\newcommand{\hc}{{\hat{\h}}}
\newcommand{\ha}{{\tilde{\h}}}
\newcommand{\Hc}{{\mathfrak{H}}}

\newcommand{\cent}{{\mathfrak{c}}}

\newcommand{\degall}[2]{{{#1}_0^{-{#2}_0-1}\rsymbol{{#1}}^{-\rsymbol{{#2}}}}}
\newcommand{\degder}[2]{{{#1}_{#2}\pder{{#1}_{#2}}}}
\newcommand{\degr}[2]{{\rsymbol{{#1}}^{-\rsymbol{{#2}}}}}
\newcommand{\degrp}[2]{{\rsymbol{{#1}}^{\rsymbol{{#2}}}}}
\newcommand{\degz}[2]{{#1}_0^{-{#2}_0-1}}
\newcommand{\Der}{\mathcal{D}}
\newcommand{\der}{D}
\newcommand{\dfunc}[3]{{{#1}^{-1}\delta\left(\frac{{#2}}{{#3}}\right)}}

\newcommand{\eltAbs}[3]{{{#1}({#2},{\rsymbol{#3}})}}
\newcommand{\eltLie}[3]{{{#1}\otimes t_0^{{#2}}\rsymbol{t}^{{\rsymbol{#3}}}}}
\newcommand{\eltzHomo}[2]{{{#1}_{-1,\rsymbol{#2}}\vac}}

\newcommand{\Espacez}[1]{{\mathcal{E}({#1})}}
\newcommand{\Espaceall}[1]{{\mathcal{E}({#1},r)}}

\newcommand{\evalueall}[4]{|_{{#1}_0={#2},\rsymbol{{#3}}={#4}}}
\newcommand{\evaluer}[2]{|_{\rsymbol{{#1}}={#2}}}
\newcommand{\expop}[2]{e^{{#1}{\pder{{#2}}}}}

\newcommand{\ky}[1]{K_Y\left({#1}\right)}

\newcommand{\lr}[1]{{L_r\left( {#1} \right)}}
\newcommand{\Lr}{{L_r}}
\newcommand{\lie}[1]{{\mathcal{L}\left( {#1} \right)}}

\newcommand{\pder}[1]{{\frac{\partial}{\partial {#1}}}}

\newcommand{\rangeall}[1]{{({#1}_0,{\rsymbol{{#1}}})\in \Zall}}
\newcommand{\ranger}[1]{\rsymbol{{#1}}\in \Zr}

\newcommand{\rsymbol}[1]{{\bf {#1}}}

\newcommand{\s}{{\mathfrak{s}}}

\newcommand{\Sc}{{S_\cent}}

\newcommand{\spaceall}[1]{\Hom \big({#1},{#1}[[\varr{x}]]((x_0))\big)}
\newcommand{\spacez}[1]{\Hom \big({#1},{#1}((x_0))\big)}
\newcommand{\spacepoly}[1]{\Hom \big({#1},{#1}((x_0)) \big)[\varr{x}]}

\newcommand{\suball}[1]{{{#1}_0,{\rsymbol{#1}}}}
\newcommand{\suballtwo}[2]{{{#1},\rsymbol{#2}}}

\newcommand{\Stva}[1]{L({#1},0)}
\newcommand{\Szero}[1]{(L({#1},0))^0}

\newcommand{\T}[1]{V\left({#1},0\right)}
\newcommand{\Tzero}[1]{\left(V\left({#1},0\right)\right)^0}
\newcommand{\Toro}{\tau}
\newcommand{\ToroD}{\tilde{\tau}}

\newcommand{\Vzero}[1]{V^0\left({#1},0\right)}
\newcommand{\VSzero}[1]{L^0\left({#1},0\right)}

\newcommand{\uea}[1]{{\mathcal{U}\left( {#1} \right)}}

\newcommand{\vac}{\rsymbol{1}}

\newcommand{\varz}[1]{{{#1}_0^{\pm 1}}}
\newcommand{\varr}[1]{{{#1}_1^{\pm 1},\dots,{#1}_r^{\pm 1}}}
\newcommand{\varvac}[1]{{[[{#1}_0,{#1}_1^{\pm 1},\dots,{#1}_r^{\pm 1}]]}}

\newcommand{\vmap}[2]{\phi^{{#1}}_{{#2}_0,\rsymbol{{#2}}}}
\newcommand{\vendo}[1]{\phi_{{#1}_0,\rsymbol{{#1}}}}

\newcommand{\Ye}[3]{{Y_{\mathcal{E}}\left({#1};{#2},{#3}\right)}}
\newcommand{\Yva}[2]{{Y^0\left({#1},{#2}\right)}}
\newcommand{\Yvamod}[3]{{Y_{{#1}}^0\left({#2},{#3}\right)}}
\newcommand{\Ylabel}[3]{{Y_{{#1}}^0\left({#2},{#3}\right)}}
\newcommand{\Ylabelnz}[3]{{Y_{{#1}}\left({#2},{#3}\right)}}
\newcommand{\Yt}[3]{{Y\left({#1};{#2},{#3}\right)}}
\newcommand{\Ytmod}[4]{{Y_{#1}\left({#2};{#3},{#4}\right)}}
\newcommand{\Ytmodlabel}[5]{{Y_{#2}^{#1}\left( {#3};{#4},{#5} \right)}}

\newcommand{\idealG}{\mathcal{G}}
\newcommand{\idealR}{\mathcal{R}}

\newcommand{\funcInd}[1]{{\rm Ind}^{#1}}
\newcommand{\funcF}[1]{{\mathcal{F}\left( {#1} \right)}}
\newcommand{\funcRes}[1]{{\rm Res}_{#1}}

\newcommand{\Alg}[1]{Algebra}
\newcommand{\alg}[1]{algebra}

\newcommand{\Subalg}[1]{Subalgebra}
\newcommand{\subalg}[1]{subalgebra}
\newcommand{\qalg}[1]{quotient algebra}
\newcommand{\Qalg}[1]{Quotient Algebra}

\newcommand{\etver}[1]{extended \tver{}}
\newcommand{\ertver}[1]{extended \rtver{}}

\newcommand{\TVer}[1]{Toroidal Vertex}
\newcommand{\Tver}[1]{Toroidal vertex}
\newcommand{\tver}[1]{toroidal vertex}
\newcommand{\rtver}[1]{$(r+1)$-\tver{}}
\newcommand{\Ver}[1]{Vertex}
\newcommand{\ver}[1]{vertex}

\newcommand{\TVA}[1]{\TVer{} \Alg{}}
\newcommand{\Tva}[1]{\Tver{} \alg{}}
\newcommand{\tva}[1]{\tver{} \alg{}}
\newcommand{\rtva}[1]{\rtver{} \alg{}}
\newcommand{\rtvsa}[1]{\rtver{} \subalg{}}
\newcommand{\rtvqa}[1]{quotient \rtver{} \alg{}}
\newcommand{\etva}[1]{\ertver{} \alg{}}
\newcommand{\va}[1]{\ver{} \alg{}}
\newcommand{\vqa}[1]{quotient \ver{} \alg{}}
\newcommand{\vs}[1]{vector space}
\newcommand{\vsa}[1]{\ver{} \subalg{}}

\newcommand{\bda}{\rsymbol{a}}
\newcommand{\x}{\rsymbol{x}}
\newcommand{\y}{\rsymbol{y}}
\newcommand{\z}{\rsymbol{z}}
\newcommand{\m}{\rsymbol{m}}
\newcommand{\n}{\rsymbol{n}}
\newcommand{\bdt}{\rsymbol{t}}

\def \<{\langle}
\def \>{\rangle}

\renewcommand{\arraystretch}{1.5}

\newcommand{\stct}[1]{satisfying the condition that}
\newcommand{\andtext}{\text{and}}
\newcommand{\fortext}{\text{for }}

\makeatletter
\@addtoreset{equation}{section}
\def\theequation{\thesection.\arabic{equation}}
\makeatother
\makeatletter

\begin{center}
{\Large \bf Simple \TVA{}s and Their Irreducible Modules}
\end{center}

\begin{center}
{Fei Kong$^{a}$, Haisheng Li$^{b}$\footnote{Partially supported
by NSA grant H98230-11-1-0161 and China NSF grant (No.11128103).},
Shaobin Tan$^{a}$\footnote{Partially supported by China NSF grant.} and Qing Wang$^{a}$\footnote{Partially supported by
 China NSF grant (No.11371024), Natural Science Foundation of Fujian Province
(No.2013J01018) and Fundamental Research Funds for the Central
University (No.2013121001).}\\
$\mbox{}^{a}$School of Mathematical Sciences, Xiamen University,
Xiamen 361005, China\\
$\mbox{}^{b}$ Department of Mathematical Sciences,\\
Rutgers University, Camden, NJ 08102, USA\\}
\end{center}

\begin{abstract} 
In this paper, we continue the study on toroidal vertex algebras initiated in \cite{LTW}, to study
concrete toroidal vertex algebras associated to toroidal Lie algebra $L_{r}(\hat{\g})=\hat{\g}\otimes L_r$,
where $\hat{\g}$ is an untwisted affine Lie algebra  and $L_r=\C\left[\varr{t}\right]$.
We first construct an \rtva{} $V(T,0)$ and show that the category of restricted $L_{r}(\hat{\g})$-modules
is canonically isomorphic to that of $V(T,0)$-modules.
Let $\cent$ denote the standard central element of $\hat{\g}$ and set $\Sc=U(L_r(\C \cent))$.
We furthermore study a distinguished  subalgebra of $V(T,0)$, denoted by $\T{\Sc}$.
We show that (graded) simple quotient toroidal vertex algebras of $V(\Sc,0)$ are parametrized
by a $\Zr$-graded ring homomorphism $\psi:\Sc\rightarrow L_r$ such that $\Image\psi$ is a $\Zr$-graded simple $\Sc$-module. Denote by $\Stva{\psi}$ the simple \rtvqa{} of $\T{\Sc}$ associated to $\psi$.
We determine for which $\psi$,  $L(\psi,0)$ is an integrable $L_{r}(\hat{\g})$-module and we then classify
irreducible $\Stva{\psi}$-modules for such a $\psi$. For our need, we also obtain various general results.
\end{abstract}

\section{Introduction}
Let $\g$ be a finite-dimensional simple Lie algebra equipped with the normalized Killing form $\<\cdot,\cdot\>$.
Let $\hat{\g}=\g\otimes \C\left[t_0^{\pm 1} \right]\oplus \C\cent$ be the untwisted affine Lie algebra.
It is well-known (see \cite{FZ}, \cite{Li1}) that there exists a canonical \va{} $V_{\hat{\g}}(\ell,0)$
associated to $\hat{\g}$ for each $\ell\in \C$
and the category of $V_{\hat{\g}}(\ell,0)$-modules is canonically isomorphic to
 the category of restricted $\hat{\g}$-modules of level $\ell$.
 Denote by $L_{\hat{\g}}(\ell,0)$ the unique graded simple quotient vertex algebra of $V_{\hat{\g}}(\ell,0)$.
 It was known (see \cite{K}) that $L_{\hat{\g}}(\ell,0)$ is an integrable $\hat{\g}$-module if and only if $\ell$
 is a non-negative integer. Furthermore,  it was known
(see \cite{FZ}, \cite{DL}, \cite{Li1}, \cite{MP1}, \cite{MP2}, \cite{DLM}) that
if $\ell$ is a non-negative integer, the category of $L_{\hat{\g}}(\ell,0)$-modules is naturally isomorphic to
the category of restricted integrable $\hat{\g}$-modules of level $\ell$ .

Toroidal Lie algebras, which are essentially central extensions of multi-loop Lie algebras,
generalizing affine Kac-Moody Lie algebras,
form a special family of infinite dimensional Lie algebras closely related to extended affine Lie algebras (see \cite{AABGP}).
A natural connection of toroidal Lie algebras with vertex algebras has also been known (see
\cite{BBS}), which uses one-variable generating functions for toroidal Lie algebras.
By considering multi-variable generating functions (cf. \cite{IKU}, \cite{IKUX}),
\tva{}s were introduced in \cite{LTW}, which generalize \va{}s in a certain natural way.

The essence of an $(r+1)$-toroidal vertex algebra $V$ is that to each vector $v\in V$,
a multi-variable vertex operator $Y(v;x_0,\x)$ is associated, which satisfies a Jacobi identity.
It is important to note that for a vertex algebra $(V,Y,\vac)$, the so-called creation property states that
$$Y(v,x)\vac\in V[[x]]\   \  \mbox{ and }\  (Y(v,x)\vac)|_{x=0}=v\   \   \   \mbox{ for }v\in V,$$
which implies that $V$ as a $V$-module is cyclic on the vacuum vector $\vac$ and
 the vertex operator map $Y(\cdot,x)$ is always injective.
However, this is {\em not} the case for an \rtva{} in general.
For an \rtva{} $V$, denote by $V^0$ the submodule of the adjoint module $V$ generated by $\vac$,
which is an \rtvsa{}.
It was proved in \cite{LTW} that $V^0$ has a canonical vertex algebra structure.
To a certain extent, $V^0$ to $V$ is the same as the core subalgebra to an extended affine Lie algebra (see \cite{AABGP}).
In this paper, we explore $V^0$ more in various directions. In particular, we show that $V^0$ is a vertex $\Z^r$-graded algebra
in a certain sense (see Section 3 for the definition).
 It is proved that if $V$ is a simple \rtva{}, then $V^0$ is also a simple \rtva{}.
Let $L$ be any quotient \rtva{} of $V$.
It is proved  (see Proposition \ref{STvaModCateProp}) that a $V$-module $W$ is naturally an $L$-module if and only if $W$ is naturally  an $L^0$-module.

In this paper, we also study $(r+1)$-toroidal vertex algebras naturally arisen from toroidal Lie algebras.
Specifically, we consider Lie algebra
\begin{eqnarray}
\Toro=\hat{\g}\otimes \C\left[ \varr{t} \right],
\end{eqnarray}
the $r$-loop algebra of an untwisted affine Lie algebra $\hat{\g}$.
We here study all possible \rtva{}s associated to $\Toro$.
The most general \rtva{} we get is $V(T,0)$, whose underlying space is
the induced $\tau$-module
\begin{eqnarray}
V(T,0)=U\big(\Toro\big)\otimes_{U(L_r(\g\otimes \C[t_0]))}T,
\end{eqnarray}
where
$T=\g\oplus\C\cent\oplus\C$ is a certain $L_r(\g \otimes \C[t_0])$-module
(see Lemma \ref{econstruction-T} and Theorem \ref{TvaDefThm} for details).
We establish an equivalence  between the category of $V(T,0)$-modules and  the category of restricted $\Toro$-modules (see Theorem \ref{TvaModThm}).

An important $(r+1)$-toroidal vertex algebra is the subalgebra $V(T,0)^0$ of $V(T,0)$, alternatively denoted by $V(\Sc,0)$,
where
$\Sc=S\left(\mathop{\bigoplus}\limits_{\ranger{m}} \C(\cent\otimes \bdt^\m)\right)$, a subalgebra of $U(\tau)$.
We furthermore study simple \rtvqa{}s of $\T{\Sc}$.
Let $\psi:\Sc\rightarrow \Lr$ be a $\Zr$-graded algebra homomorphism such that $\Image\psi$ is a $\Zr$-graded simple $\Sc$-module (following Rao \cite{R1}),
and let $I(\psi)$ be the ideal of $\T{\Sc}$ generated by $\Ker \psi$.
Denote by $\T{\psi}$ the \rtvqa{} of $\T{\Sc}$ modulo $I(\psi)$.
It is proved that $\T{\psi}$ has a unique graded simple \rtvqa{}, which is denoted by $\Stva{\psi}$,
and these toroidal vertex algebras $\Stva{\psi}$ exhaust all graded simple \rtvqa{}s of $\T{\Sc}$.
Among the main results, we prove that $L(\psi,0)$ is an integrable $\tau$-module if and only if $\psi$ is given by
\begin{eqnarray}\label{IntroEq1}
\psi(\cent\otimes \bdt^\m)=\left(\sum\limits_{i=1}^s \ell_i \bda_i^\m\right) \bdt^{\m} \  \  \  \  (\m\in \Z^r)
\end{eqnarray}
for some finitely many positive integers $\ell_1,\dots,\ell_s$
and some distinct elements $\bda_1,\dots,\bda_s\in (\C^\times)^r$ ($\C^\times = \C\setminus \{0\}$).
Furthermore, assuming that $L(\psi,0)$ is an integrable $\tau$-module,
we find a necessary and sufficient condition that a $\Toro$-module is an $\Stva{\psi}$-module,
and moreover, we classify all irreducible $\Stva{\psi}$-modules.

 For this paper, determining the integrability of  $L(\psi,0)$ and its irreducible modules is the core.
  One of the difficult problems is to find the explicit characterization of $L(\psi,0)$
 in terms of $V(\psi,0)$. (Compared with the case for the ordinary affine vertex algebras, this problem is much harder.)
To achieve these goals, we extensively use a graded simple quotient
vertex algebra $L^{0}(\psi,0)$ of $L(\psi,0)$ and we show that $L(\psi,0)$ can be  embedded  canonically into
the tensor product vertex algebra $L^{0}(\psi,0)\otimes L_r$, which is naturally a toroidal vertex algebra.
As an important step, by using a result of Rao we show that for a $\psi$ given by (\ref{IntroEq1}),
$L^{0}(\psi,0)$ is isomorphic to the tensor product of simple vertex algebras $L_{\hat{\g}}(\ell_i,0)$ for $i=1,\dots,s$.
Eventually, we achieve a classification of irreducible  $L(\psi,0)$-modules in terms of irreducible
$L_{\hat{\g}}(\ell_i,0)$-modules.

We mention that the present work is closely related to some of Rao's.
Let $\ToroD$ be the semi-direct product of $\Toro$ and the degree derivations $d_0,\dots,d_r$.
 In \cite{R2}, Rao classified  irreducible integrable $\ToroD$-modules with finite dimensional weight spaces.
It was implicitly proved therein that any irreducible integrable restricted $\Toro\oplus \C d_0$-module
with finite dimensional weight spaces is isomorphic to a concrete module associated to
some finitely many integrable highest weight $\hat{\g}$-modules
and some distinct vectors in $(\C^\times)^r$ of the same number.
When $L(\psi,0)$ is an integrable $\tau$-module,  irreducible $L(\psi,0)$-modules are closely related to these $\tau$-modules.
Rao's results are very helpful at certain stages of this work.

This paper is organized as follows:
In Section 2, after recall the basics about a general \rtva{} $V$ and  its subalgebra $V^0$,
we study the connection between the module category of $V$ and that of $V^0$
viewed as  an \rtva{} and as a vertex algebra.
In Section 3,  we study \rtva{} $\lr{A}=A\otimes \C\left[\varr{t}\right]$ associated to a \va{} $A$ and
we relate an \rtva{} $V$ to  $\lr{A}$ for some vertex algebra $A$.
In Section 4, by using toroidal Lie algebra $\Toro$ we construct and study \rtva{}s $\T{T}$, $V(\Sc,0)$, $V(\psi,0)$, and $L(\psi,0)$.
In Section 5, we determine when $\Stva{\psi}$ viewed as a $\Toro$-module is integrable.
In Section 6, we classify irreducible modules for $\Stva{\psi}$ with $\Stva{\psi}$ an integrable $\Toro$-module.

\section{ \TVA{}s and Vertex Algebras}
In this section, we first recall some basic definitions and then study ideals of a \tva{},
simple \tva{}s and their module categories.

We begin by recalling the notion of \tva{}.
Let $r$ be a positive integer.
For a \vs{} $W$, set
\[
\Espaceall{W}=\spaceall{W}
\]
and
\[
\Espacez{W}=\spacez{W}.
\]
Set $\x=(x_1,\dots,x_r)$. For $\m\in \Z^{r}$, set
$$\x^\m=x_1^{m_1}\cdots x_r^{m_r}.$$

The following notion was introduced in \cite{LTW}:

\begin{de}
An {\em $(r+1)$-\tva{}} is a \vs{} $V$, equipped with a linear map
\begin{eqnarray*}
\Yt{\cdot}{x_0}{\x}:&& V\rightarrow \Espaceall{V},\\
					&& v\mapsto \Yt{v}{x_0}{\x}=\sum\limits_{\rangeall{m}}v_\suball{m}\degall{x}{m}
\end{eqnarray*}
and equipped with a vector $\vac\in V$, satisfying the conditions that
\[
\begin{array}{llll}
\Yt{\vac}{x_0}{\x}v=v & \andtext &\Yt{v}{x_0}{\x}\vac\in V\varvac{x} & \fortext v\in V,
\end{array}
\]
and that for $u,v\in V$,
\begin{eqnarray}\label{jacobi-tva}
&& \dfunc{z_0}{x_0-y_0}{z_0}\Yt{u}{x_0}{\z\y}\Yt{v}{y_0}{\y}
		-\dfunc{z_0}{y_0-x_0}{-z_0}\Yt{v}{y_0}{\y}\Yt{u}{x_0}{\z\y}\ \ \ \nonumber\\
&&		\hspace{3cm}=\dfunc{y_0}{x_0-z_0}{y_0}\Yt{\Yt{u}{z_0}{\z}v}{y_0}{\y},
\end{eqnarray}
where
\[
\Yt{u}{x_0}{\z\y}=\sum\limits_{\rangeall{m}}u_\suball{m}\degz{x}{m}\degr{z}{m}\degr{y}{m}.
\]
\end{de}

For a given  $(r+1)$-\tva{} $V$, the notion of \rtvsa{} and that of ideal are defined in the obvious way.
On the other hand,  the notion of $V$-module is also defined in the obvious way;
a {\em $V$-module} is a vector space $W$ equipped with a linear map
$Y_{W}(\cdot; x_0,\x)$ from $V$ to $\Espaceall{W} $ such that
$$Y_{W}({\bf 1};x_0,\x)=1_{W}\   \  (\mbox{the identity operator on }W)$$
and such that for any $u,v\in V$, the Jacobi identity (\ref{jacobi-tva}) with $Y$ replaced by $Y_{W}$
in five obvious places holds.

{\em From now on, we fix the positive integer $r$ and we simply call an $(r+1)$-\tva{} a \tva..}

\begin{de}
Let $V$ be a toroidal vertex algebra. A {\em derivation} of $V$ is a linear operator $D$ on $V$, \stct{}
\begin{equation}
\begin{array}{llll}
D(\vac)=0& \andtext &\left[D,\Yt{v}{x_0}{\x}\right]=\Yt{D(v)}{x_0}{\x} &\fortext v\in V.
\end{array}
\end{equation}
\end{de}

\begin{de}
An {\em extended \tva{}} is a \tva{} $V$ equipped with derivations $\Der_0,\Der_1,\dots,\Der_r$ such that
\begin{eqnarray}
&&\Yt{\Der_0(v)}{x_0}{\x}=\pder{x_0}\Yt{v}{x_0}{\x},\nonumber\\
&&\Yt{\Der_j(v)}{x_0}{\x}=\left(\degder{x}{j}\right)\Yt{v}{x_0}{\x}
\end{eqnarray}
for $v\in V$, $1\leq j\leq r$.
\end{de}

\begin{rem} \label{difference}
{\em Note that in the definition of a vertex algebra $V$,
the so-called creation property states that for every $v\in V$,
$$Y(v,x){\bf 1}\in V[[x]]\   \  \mbox{and } \  \lim_{x\rightarrow 0}Y(v,x){\bf 1}=v.$$
It follows that the adjoint module $V$ is cyclic on the vacuum vector ${\bf 1}$ and
the vertex operator map $Y(\cdot,x)$ of the vertex algebra is {\em always} injective.
In contrast to this, for a  \tva{} $V$, while such a creation property is missing,
$V$ as a $V$-module may be not cyclic on ${\bf 1}$, and the vertex operator map $\Yt{\cdot}{x_0}{\x}$ may be not injective. }
\end{rem}

The following result was obtained in \cite{LTW}:

\begin{prop}\label{TvaStructProp}
Let $V$ be a \tva{}. Set
\begin{eqnarray}
V^0={\rm span}\{ v_\suball{m}\vac\ |\ v\in V,\ \rangeall{m} \}.
\end{eqnarray}
Then $V^0$ is a toroidal vertex subalgebra of $V$,
\begin{eqnarray}\label{V0-property1}
\Yt{v}{x_0}{\x}\in \Espacez{V}[\varr{x}]\  \  \  \mbox{ for }v\in V^0,
\end{eqnarray}
and $\Yt{\cdot}{x_0}{\x}$ is injective on $V^0$.
Define a linear map $\Yva{\cdot}{x_0}:V^0\rightarrow \Espacez{V}$ by
\begin{eqnarray}
\Yva{v}{x_0}=\Yt{v}{x_0}{\x}\evaluer{\x}{1} \  \  \  \fortext \,v\in V^0.
\end{eqnarray}
Then
$\left( V^0,Y^0,\vac\right)$ carries the structure of a \va{} and $(V,Y^0)$ carries the structure of a $V^0$-module.
Furthermore,  for $v\in V$, $\rangeall{m}$,
\begin{eqnarray}
&&\Yt{v_\suball{m}\vac}{x_0}{\x}=Y^0(v_{m_0,\m}{\bf 1},x_0)\degr{x}{m}\in \degr{x}{m}\Espacez{V},\label{etva-m}\\
&&\Yt{v}{x_0}{\x}=\sum\limits_{\ranger{m}}\Yt{\eltzHomo{v}{m}}{x_0}{\x}.\label{V0-sum}
\end{eqnarray}
\end{prop}

We have the following analogue for $V$-modules:

\begin{prop}\label{TvaModStructProp}
Let $V$ be a \tva{} and let $\left(W, Y_W\right)$ be a $V$-module.
Then for $v\in V$, we have
\begin{eqnarray}\label{eV0-module}
\Ytmod{W}{v_\suball{m}\vac}{x_0}{\x}\in \degr{x}{m}\Espacez{W} & \fortext \rangeall{m}
\end{eqnarray}
and
\begin{equation}\label{ev1sum}
\Ytmod{W}{v}{x_0}{\x}=\sum\limits_{\ranger{m}}\Ytmod{W}{\eltzHomo{v}{m}}{x_0}{\x}.
\end{equation}
For $u\in V^0$, set
\begin{eqnarray}
Y_{W}^0(u,x_{0})=Y_{W}(u;x_{0},\x)|_{\x=1}\in \Espacez{W}.
\end{eqnarray}
Then $(W,Y_{W}^0)$ carries the structure of a module for $V^0$ viewed as a vertex algebra.
Furthermore,  if $(W,Y_{W})$ is irreducible,  $(W,Y_{W}^0)$ is  irreducible.
\end{prop}

\begin{proof}
For $v\in V$, with $Y_{W}({\bf 1};x_0,\x)=1_{W}$, from the Jacobi identity we have
\begin{eqnarray}\label{2.3}
&&\Ytmod{W}{\Yt{v}{z_0}{\z}\vac}{y_0}{\y}\nonumber\\
&=&\Res_{x_0}\dfunc{z_0}{x_0-y_0}{z_0}\Ytmod{W}{v}{x_0}{\z\y}
		-\Res_{x_0}\dfunc{z_0}{y_0-x_0}{-z_0}\Ytmod{W}{v}{x_0}{\z\y} \nonumber\\
&=&\Res_{x_0}	\dfunc{x_0}{y_0+z_0}{x_0}\Ytmod{W}{v}{x_0}{\z\y}\nonumber\\	
&=&\Ytmod{W}{v}{y_0+z_0}{\z\y} \nonumber\\
&=&\expop{z_0}{y_0}\Ytmod{W}{v}{y_0}{\z\y}.
\end{eqnarray}
Let $(m_0,\m)\in \Z\times \Z^{r}$. We know $v_{m_0,\m}{\bf 1}=0$ if $m_0\ge 0$. As for $m_0<0$, from (\ref{2.3})  we get
\begin{eqnarray}\label{2.4}
\Ytmod{W}{v_{-k-1,\m}{\bf 1}}{y_0}{\y}&=& \frac{1}{k!}\left(\frac{\partial}{\partial y_0}\right)^k \sum_{m_0\in \Z}v_{m_0,\m}y_0^{-m_0-1}\degr{y}{m}
\nonumber\\
&=&\sum_{m_0\in \Z}\binom{-m_0-1}{k}v_{m_0,\m}y_0^{-m_0-k-1}\degr{y}{m}
\end{eqnarray}
for $k\geq 0$. This proves the first part.
Furthermore, using (\ref{2.4}) we obtain
\[
\Ytmod{W}{v}{y_0}{\y}=\sum\limits_{\rangeall{m}}v_{m_0,\m}y_0^{-m_0-1}\y^{-\m}
= \sum\limits_{\ranger{m}} Y_{W}(v_{-1,\m}{\bf 1};y_0,\y),
\]
which, though an infinite sum,  exists in $\Espaceall{V}$ due to (\ref{eV0-module}).
The proof for the other assertions is straightforward.
\end{proof}

The following basic facts immediately follow
from Proposition \ref{TvaModStructProp} and its proof:

\begin{lem}\label{simple-facts}
Let $V$ be a \tva{}. For $v\in V$, $(m_0,\m),\ (n_0,\n)\in \Z\times \Z^r$, we have
\begin{eqnarray}\label{e2.14}
(v_{m_0,\m}{\bf 1})_{n_0,\n}=0\   \   \  \mbox{whenever }m_0\ge 0\   \mbox{or }\m\ne \n
\end{eqnarray}
and
\begin{eqnarray}\label{2.15}
(v_{-k-1,\m}{\bf 1})_{n_0,\m}=\binom{k-n_0-1}{k}v_{n_0-k,\m}
\end{eqnarray}
on any $V$-module for $k\ge 0$.
In particular, we have
\begin{eqnarray}\label{e2.16}
(v_{-1,\m}{\bf 1})_{-1,\m}{\bf 1}=v_{-1,\m}{\bf 1}.
\end{eqnarray}
\end{lem}

As immediate consequences we have:

\begin{coro}\label{rV0fact}
Let $V$ be a toroidal vertex algebra. Then
\begin{equation}\label{ZeroSpanEq}
V^0 = {\rm span}\{ \eltzHomo{v}{m}\,|\,v\in V,\  \m\in \Z^{r} \},
\end{equation}
which also equals ${\rm span}\{ \eltzHomo{v}{m}\,|\,v\in V^0,\  \m\in \Z^{r} \}$.
\end{coro}

\begin{coro}\label{V0-characterization}
Let $V$ be a toroidal vertex algebra and let $v\in V$. Then $v\in V^0$ if and only if
\begin{eqnarray*}
&&v_{m_0,\m}=0\  \  \  \mbox{ for all }m_{0}\in \Z  \mbox{ and for all but finitely many }\m\in \Z^{r},\\
&&v=\sum_{\m\in \Z^{r}}v_{-1,\m}{\bf 1}.
\end{eqnarray*}
\end{coro}


For $v\in V,\ \m\in \Z^{r}$, we set
\begin{eqnarray}
Y(v;x_0,\m)=\sum_{m_0\in \Z}v_{m_0,\m}x_0^{-m_0-1}.
\end{eqnarray}

The following is a technical result which we shall need:

\begin{lem} For $a,b\in V$, $m_0\in \Z,\ \m,\n\in \Z^{r}$, we have
\begin{eqnarray}\label{tempEq1}
(a_{m_0,\m}b)_{-1,\m+\n}{\bf 1}=a_{m_0,\m}b_{-1,\n}{\bf 1}.
\end{eqnarray}
\end{lem}

\begin{proof} From the Jacobi identity for the pair $(a,b)$ we get
\begin{eqnarray*}
&&\Yt{\Yt{a}{z_0}{\m}b}{y_0}{\m+\n}\vac \nonumber\\
&=&\Res_{x_0}\dfunc{z_0}{x_0-y_0}{z_0}\Yt{a}{x_0}{\m}\Yt{b}{y_0}{\n}\vac\nonumber \\
&&-\Res_{x_0}\dfunc{z_0}{y_0-x_0}{-z_0}\Yt{b}{y_0}{\n}\Yt{a}{x_0}{\m}\vac\nonumber \\
&=&\Res_{x_0}\dfunc{z_0}{x_0-y_0}{z_0}\Yt{a}{x_0}{\m}\Yt{b}{y_0}{\n}\vac\nonumber \\
&=&\Yt{a}{z_0+y_0}{\m}\Yt{b}{y_0}{\n}\vac.
\end{eqnarray*}
Setting $y_0=0$ and then applying $\Res_{z_0}z_0^{m_0}$ we obtain (\ref{tempEq1}).
\end{proof}

Recall that for a \tva{} $V$, $V^0$ is a \tva{} and  a \va{} as well.
Unless noted otherwise, {\em  a $V^0$-module always stands for a module for $V^0$ viewed as a toroidal vertex algebra.}
Denote by $\Mod V^0$ and $\Mod^0 V^0$ the module categories of $V^0$ viewed as a \tva{} and as a vertex algebra, respectively.

The following is a connection between $V$-modules and $V^{0}$-modules:

\begin{prop}\label{prop:Vmod-V0mod}
Let $V$ be a  \tva{}. If  $(W,Y_{W})$ is  a $V$-module, then $(W,Y_{W}')$ is
 a $V^0$-module  where $Y_W'$ is the restriction of $Y_W$,
 satisfying the condition that
 for every $v\in V,\ w\in W$, there exists an integer $k$ such that
\begin{eqnarray}\label{Y0-toroidal-truncation}
x_0^{k}Y_{W}'\left(v_{-1,\m}{\bf 1};x_0,\x\right)w\in \x^{-\m}W[[x_0]]\  \  \  \mbox{ for all }\m\in \Z^r.
\end{eqnarray}
 Conversely, if
$(W,Y_{W}')$ is a $V^0$-module  satisfying the condition above, then $(W,Y_W)$ is a $V$-module where
\begin{eqnarray}\label{toroidal-connection}
Y_{W}(v;x_0,\x)=\sum_{\m\in \Z^r}Y_{W}'(v_{-1,\m}{\bf 1};x_0,\x)\  \  \   \mbox{ for }v\in V.
\end{eqnarray}
\end{prop}

\begin{proof} Suppose that $(W,Y_{W})$ is a $V$-module.
 Then $(W,Y_{W}')$ is  automatically a $V^0$-module. For $v\in V,\ w\in W$, from Proposition \ref{TvaModStructProp},
 we have $Y_{W}(v_{-1,\m}{\bf 1};x_0,\x)w\in \x^{-\m}W((x_0))$ for $\m\in \Z^r$ and
 $$Y_{W}(v;x_0,\x)w=\sum_{\m\in \Z^r}Y_{W}(v_{-1,\m}{\bf 1};x_0,\x)w. $$
 Let $k$ be an integer such that $x_0^kY_{W}(v;x_0,\x)w\in W[[\x^{\pm 1}]][[x_{0}]]$.
Then (\ref{Y0-toroidal-truncation}) follows.

Conversely, assume that $(W,Y_W')$ is a $V^0$-module satisfying the very condition.
Define a vertex operator map $Y_{W}$ on $V$ by using (\ref{toroidal-connection}). Note that
for $v\in V$, we have $Y_{W}(v;x_0,\x)\in \Espaceall{W}$ due to condition (\ref{Y0-toroidal-truncation}).
It is clear that $Y_{W}({\bf 1};x_0,\x)=1_{W}$.
Let $u,v\in V,\ m_0\in \Z,\ \m,\n, {\bf p}\in \Z^r$. Using Lemma \ref{simple-facts} and (\ref{tempEq1}) we get
\begin{eqnarray}\label{tempEq2}
(u_{-1,\m}{\bf 1})_{m_0,{\bf p}}(v_{-1,\n}{\bf 1})=\delta_{\m,{\bf p}}u_{m_0,\m}v_{-1,\n}{\bf 1}
=\delta_{\m,{\bf p}}(u_{m_0,\m}v)_{-1,\m+\n}{\bf 1}.
\end{eqnarray}
That is,
\begin{eqnarray}\label{etva-fact2}
Y(u_{-1,\m}{\bf 1};z_0,\z)(v_{-1,\n}{\bf 1})=\Res_{x_0}x_0^{-1}\z^{-\m}Y(Y(u;z_0,\m)v;x_{0},\m+\n){\bf 1}.
\end{eqnarray}
Using the Jacobi identity for $Y_{W}'$ and this relation, we obtain
\begin{eqnarray*}
&&\dfunc{z_0}{x_0-y_0}{z_0}\Ytmod{W}{u}{x_0}{\z\y}\Ytmod{W}{v}{y_0}{\y}		\\
&&\ \ \ \  -\dfunc{z_0}{y_0-x_0}{-z_0}\Ytmod{W}{v}{y_0}{\y}\Ytmod{W}{u}{x_0}{\z\y} \\
&=&\dfunc{z_0}{x_0-y_0}{z_0}\sum\limits_{\ranger{m}}\sum\limits_{\ranger{n}}
			Y_{W}'(u_{-1,\m}{\bf 1};x_0,\z\y)Y_{W}'(v_{-1,\n}{\bf 1};y_0,\y)\\
&& -\dfunc{z_0}{y_0-x_0}{-z_0}\sum\limits_{\ranger{m}}\sum\limits_{\ranger{n}}
		Y_{W}'(v_{-1,\n}{\bf 1};y_0,\y)Y_{W}'(u_{-1,\m}{\bf 1};x_0,\z\y)	 \\
&=&\dfunc{y_0}{x_0-z_0}{y_0}\sum\limits_{\ranger{m}}\sum\limits_{\ranger{n}}
			Y_{W}'\left(Y(u_{-1,\m}{\bf 1};z_0,\z)(v_{-1,\n}{\bf 1});y_0,\y\right) \\
&=&\dfunc{y_0}{x_0-z_0}{y_0}\sum\limits_{\ranger{m}}\sum\limits_{\ranger{n}}
			\Res_{x_0}x_0^{-1}\z^{-\m}Y_{W}'\left( Y(Y(u;z_0,\m)v;x_0,\m+\n)\vac;y_0,\y\right)\\
&=&\dfunc{y_0}{x_0-z_0}{y_0}\Ytmod{W}{\Yt{u}{z_0}{\z}v}{y_0}{\y}.		
\end{eqnarray*}
This proves that $(W,Y_{W})$ is a $V$-module.
\end{proof}

As one of the main results of this section, we have:

\begin{thm}\label{ZeroModCateProp}
Let $V$ be a  \tva{}. If  $(W,Y_{W})$ is  a $V^0$-module, then $(W,Y_{W}^0)$ is
 a module for $V^0$ viewed as a vertex algebra, where
 \begin{eqnarray*}
Y_{W}^0(u,x_0)= \Ytmod{W}{u}{x_0}{\x}|_{\x=1}\   \  \   \mbox{ for }u\in V^0.
\end{eqnarray*}
Conversely, if
$(W,Y_{W}^0)$ is a module for $V^0$ viewed as a vertex algebra, then $(W,Y_W)$ is a $V^0$-module where
\begin{eqnarray}\label{connection}
Y_{W}(u;x_0,\x)=\sum_{\m\in \Z^r}Y_{W}^0(u_{-1,\m}{\bf 1},x_0)\x^{-\m}\  \  \   \mbox{ for }u\in V^0.
\end{eqnarray}
Furthermore, category $\Mod V^0$ is naturally isomorphic to category $\Mod^0 V^0$.
\end{thm}

\begin{proof} Let $(W,Y_W)$ be a $V^0$-module.
By Proposition \ref{TvaModStructProp}, for $u\in V^0$, we have
$$Y_{W}(u;x_0,\x)\in \Espacez{W}[x_{1}^{\pm 1},\dots,x_{r}^{\pm 1}],$$
so that $Y_{W}(u;x_0,\x)|_{\x=1}$ exists in $\Espacez{W}$. In view of this,
$\Yvamod{W}{\cdot}{x_0}$  is a well defined linear map from $V^0$ to $\Espacez{W}$.
It can be readily seen that  $(W,Y_W^0)$ is a module for $V^0$ viewed as a \va{}.
Let $\alpha:W_{1}\rightarrow W_{2}$ be a  homomorphism of $V_0$-modules.
For  $v\in V^0$, $w\in W_{1}$, we have
\begin{eqnarray*}
\alpha(\Yvamod{W_{1}}{v}{x_0}w)=\alpha\left( \Ytmod{W_{1}}{v}{x_0}{\x}w|_{\x=1} \right)=\Ytmod{W_{2}}{v}{x_0}{\x}\alpha\left( w \right)|_{\x=1}=\Yvamod{W_{2}}{v}{x_0}\alpha(w).
\end{eqnarray*}
This shows that $\alpha$ is also a homomorphism of modules for $V^0$ viewed as a \va{}.
In this way, we get a functor $\mathcal{A}$ from $\Mod V^0$ to $\Mod^0 V^0$.

On the other hand, let $(W,Y_W^0)$ be a module for $V^0$ viewed as a \va{}. Define a linear map
$\Ytmod{W}{\cdot}{x_0}{\x}: \  \ V^0 \rightarrow \Espaceall{W}$ by
\begin{eqnarray}\label{tempDefEq1}
Y_{W}(u;x_0,\x)= \sum\limits_{\ranger{m}}\Yvamod{W}{\eltzHomo{u}{m}}{x_0}\degr{x}{m}
\  \  \  \mbox{ for }u\in V^0.
\end{eqnarray}
Note that  $u_{-1,\m}{\bf 1}=0$ for all but finitely many $\m\in \Z^{r}$ by Corollary \ref{V0-characterization}, so that
such defined $Y_{W}(u;x_0,\x)$ indeed lies in $\Espaceall{W}$.
Let $u,v\in V^0$. Using the Jacobi identity for $Y_{W}^0$ and (\ref{etva-fact2}),  we obtain
\begin{eqnarray*}
&&\dfunc{z_0}{x_0-y_0}{z_0}\Ytmod{W}{u}{x_0}{\z\y}\Ytmod{W}{v}{y_0}{\y}		\\
&&\ \ \ \  -\dfunc{z_0}{y_0-x_0}{-z_0}\Ytmod{W}{v}{y_0}{\y}\Ytmod{W}{u}{x_0}{\z\y} \\
&=&\dfunc{z_0}{x_0-y_0}{z_0}\sum_{\m,\n\in \Z^r}
			\Yvamod{W}{\eltzHomo{u}{m}}{x_0}\Yvamod{W}{\eltzHomo{v}{n}}{y_0}\degr{z}{m}\y^{-\m-\n} \\
&& -\dfunc{z_0}{y_0-x_0}{-z_0}\sum_{\m,\n\in \Z^r}
			\Yvamod{W}{\eltzHomo{v}{n}}{y_0}\Yvamod{W}{\eltzHomo{u}{m}}{x_0}\degr{z}{m}\y^{-\m-\n} \\
&=&\dfunc{y_0}{x_0-z_0}{y_0}\sum_{\m,\n\in \Z^r}
			\Yvamod{W}{\Yva{\eltzHomo{u}{m}}{z_0}\big(\eltzHomo{v}{n}\big)}{y_0}\degr{z}{m}\y^{-\m-\n} \\
	&=&\dfunc{y_0}{x_0-z_0}{y_0}\sum_{\m,\n\in \Z^r}
			\Yvamod{W}{Y(u_{-1,\m}{\bf 1};z_0,\z)\big(\eltzHomo{v}{n}\big)}{y_0}\y^{-\m-\n} \\		
&=&\dfunc{y_0}{x_0-z_0}{y_0}\Res_{x_0}x_0^{-1}\sum_{\m,\n\in \Z^r}
			\Yvamod{W}{Y\Big( Y(u;z_0,\m)v;x_0, \m+\n\Big)\vac}{y_0}\z^{-\m}\y^{-\m-\n} \\
&=&\dfunc{y_0}{x_0-z_0}{y_0}\sum_{\m\in \Z^r}
			Y_{W}\left(Y(u;z_0,\m)v;y_0,\y\right)\z^{-\m} \\			
&=&\dfunc{y_0}{x_0-z_0}{y_0}Y_{W}\left(Y(u;z_0,\z)v;y_0,\y\right).		
\end{eqnarray*}
This shows that  that $(W,Y_W)$ is a $V^0$-module.

Let $\alpha:W\rightarrow W'$ be a homomorphism of modules for $V^0$ viewed as a \va{}.
For any $v\in V,\ \m\in \Z^r$, from (\ref{tempDefEq1}), using (\ref{eV0-module}) and (\ref{ev1sum})  we have
\begin{eqnarray*}
&&\Ytmod{W}{\eltzHomo{v}{m}}{x_0}{\x}=\Yvamod{W}{\eltzHomo{v}{m}}{x_0}\degr{x}{m},\\
&&\Ytmod{W'}{\eltzHomo{v}{m}}{x_0}{\x}=\Yvamod{W'}{\eltzHomo{v}{m}}{x_0}\degr{x}{m}.
\end{eqnarray*}
It then follows that for any $v\in V$, $w\in W,$ and $\m\in \Z^r$,
\begin{eqnarray*}
\alpha\left(\Ytmod{W}{\eltzHomo{v}{m}}{x_0}{\x}w\right)=\Ytmod{W'}{\eltzHomo{v}{m}}{x_0}{\x}\alpha(w).
\end{eqnarray*}
This together with Corollary \ref{rV0fact} implies that $\alpha$ is also a $V^0$-module homomorphism.
In this way, we get a functor $\mathcal{B}$ from $\Mod^0 V^0$ to $\Mod V^0$.
It is straightforward to see that $\mathcal{A}$ is an isomorphism of categories
with $\mathcal{B}$ as its inverse.
\end{proof}

As an immediate consequence of Proposition \ref{TvaStructProp} we have:

\begin{coro}\label{V=V0}
Let $V$ be a \tva{}. Then $V^0=V$ if and only if $Y(\cdot;x_0,\x)$ on $V$ is injective and
\begin{eqnarray}\label{efinite}
Y(v;x_0,\x){\bf 1}\in V[[x_{0}]][x_{1}^{\pm 1},\dots,x_{r}^{\pm 1}]\  \  \   \mbox{ for all }v\in V.
\end{eqnarray}
\end{coro}

Furthermore,  using Corollary \ref{V=V0} and Proposition \ref{TvaStructProp} we immediately have:

\begin{coro}
Let $V$ be a \tva{}. Then $(V^0)^0=V^0$.
\end{coro}

The following is straightforward to prove:

\begin{lem}\label{Kylemma}
Let $V$ be a \tva{}. Set
\begin{equation}
\ky{V}=\{ v\in V\,|\, \Yt{v}{x_0}{\x}=0 \}.
\end{equation}
Then
\begin{equation*}
\ky{V}=\{ v\in V\,|\,\Yt{v}{x_0}{\x}\vac=0 \},
\end{equation*}
$\ky{V}$ is an ideal of $V$, and  $\der\left(\ky{V}\right)\subset \ky{V}$ for any derivation $\der$ of $V$.
On the other hand, for any $V$-module $(W,Y_W)$, we have
\[
\Ytmod{W}{v}{x_0}{\x}=0\  \  \   \mbox{ for }v\in \ky{V}.
\]
\end{lem}

\begin{de}
For a \tva{} $V$, we define
\begin{eqnarray}\label{BarDefEq123}
\overline{V}=V/\ky{V},
\end{eqnarray}
a quotient toroidal vertex algebra.
\end{de}

\begin{rem}\label{eVVbar}
{\em Note that for any \tva{} $V$, the quotient map from $V$ to $\overline{V}$ naturally gives rise to a functor $\eta$ from
the category of $\overline{V}$-modules to the category of $V$-modules. In view of Lemma \ref {Kylemma},
$\eta$ is an isomorphism.}
\end{rem}

\begin{lem} \label{barV=V}
Suppose that $V$ is a \tva{} such that
\begin{eqnarray}\label{efinite}
Y(v;x_0,\x){\bf 1}\in V[[x_{0}]][x_{1}^{\pm 1},\dots,x_{r}^{\pm 1}]\  \  \   \mbox{ for all }v\in V.
\end{eqnarray}
Then $\overline{V}^{0}=\overline{V}$.
\end{lem}

\begin{proof} Let $v\in V$. From our assumption, $v_{-1,\m}{\bf 1}=0$ for all but finitely many $\m\in \Z^{r}$.
On the other hand, by Proposition \ref{TvaStructProp}  we have
$$Y(v;x_0,\x)=\sum_{\m\in \Z^{r}}Y(v_{-1,\m}{\bf 1};x_0,\x)\   \mbox{ on }V.$$
Then $v-\sum_{\m\in \Z^{r}}v_{-1,\m}{\bf 1}\in K_{Y}(V)$. It follows that $\overline{V}\subset \overline{V}^{0}$.
Thus $\overline{V}^{0}=\overline{V}$.
\end{proof}

Next, we study the relation between the ideals of $V$ and $V^0$.

\begin{lem}\label{IdealCorresLem}
Let $V$ be a \tva{}.
For any ideal $I$ of $V$, set
\begin{eqnarray}
&&\idealR(I)=I\cap V^0.
\end{eqnarray}
Then $\idealR(I)$ is an ideal of $V^0$ and
\begin{eqnarray}\label{IdealCorresLemEqtemp1}
\idealR(I)={\rm span} \{ a_\suball{m}\vac\  |\   a\in I,\  \rangeall{m} \}.
\end{eqnarray}
On the other hand, for any ideal $I^0$ of $V^0$, set
\begin{eqnarray}
&&\idealG(I^0)={\rm span} \{ a_\suball{m}v\  |\  a\in I^0,\   v\in V,\  \rangeall{m} \}.
\end{eqnarray}
Then $\idealG(I^0)$ is an ideal of $V$ and it is the ideal generated by $I^0$.
\end{lem}

\begin{proof}
It is clear that $\idealR(I)$ is an ideal of $V^0$.
Recall from Corollary \ref{V0-characterization}  that for  $a\in V^0$,
$a=\sum\limits_{\ranger{m}}\eltzHomo{a}{m},$
which is a finite sum. Using this we get
$$\idealR(I)=I\cap V^0\subset{\rm span}\{ a_\suball{m}\vac\  |\  a\in I,\  \rangeall{m} \}\subset I\cap V^0=\idealR(I),$$
proving (\ref{IdealCorresLemEqtemp1}).

Now, let $I^0$ be an ideal of $V^0$.
For $a\in I^0\  (\subset V^0)$, from Corollary \ref{V0-characterization}, we have
$$a=\sum\limits_{\ranger{m}} \eltzHomo{a}{m}\  \  (\mbox{a finite sum}) \in \idealG(I^0).$$
This proves $I^0\subset\idealG(I^0)$.
Let  $a\in I^0$, $u,v\in V,\ \m\in \Z^{r}$. Then there exists a nonnegative integer $k$ such that
\begin{eqnarray}\label{locality-need}
(x_0-y_0)^{k}\Yt{u}{x_0}{\m}\Yt{a}{y_0}{\y}v =(x_0-y_0)^{k}\Yt{a}{y_0}{\y}\Yt{u}{x_0}{\m}v.
\end{eqnarray}
Noticing that the right-hand side lies in $\idealG(I^0)[[x_0^{\pm 1},y_{i}^{\pm 1}\  | \  0\le i\le r]]$, we get
$$(x_0-y_0)^{k}\Yt{u}{x_0}{\m}\Yt{a}{y_0}{\y}v\in \left(\idealG(I^0)[[y_{i}^{\pm 1}\  | \  1\le i\le r]]\right)((x_0))((y_0)).$$
Multiplying by $(x_0-y_0)^{-k}$ we obtain
$$\Yt{u}{x_0}{\m}\Yt{a}{y_0}{\y}v\in \left(\idealG(I^0)[[y_{i}^{\pm 1}\  | \  1\le i\le r]]\right)((x_0))((y_0)).$$
Then it follows that
 $\idealG(I^0)$ is a left ideal of $V$.
Furthermore, for  $a\in I^0$, $u,v\in V,\ \m\in \Z^{r}$, we have
\begin{eqnarray*}
	&&\Yt{\Yt{a}{z_0}{\m}u}{y_0}{\y}v\\
	&=&\Res_{x_0}\dfunc{z_0}{x_0-y_0}{z_0}\degr{y}{m}\Yt{a}{x_0}{\m}\Yt{u}{y_0}{\y}v \\
	&&-\Res_{x_0}\dfunc{z_0}{y_0-x_0}{-z_0}\degr{y}{m}\Yt{u}{y_0}{\y}\Yt{a}{x_0}{\m}v.
\end{eqnarray*}
From this it follows that $\idealG(I^0)$ is also a right ideal of $V$.
Therefore, $\idealG(I^0)$ is an ideal of $V$. It is clear that $\idealG(I^0)$ is the smallest ideal containing $I^0$.
\end{proof}

Furthermore, we have:

\begin{prop}\label{IdealCorresProp}
Let $V$ be a \tva{}. Then $\idealR(\idealG(I^0))=I^0$ for any ideal $I^0$ of $V^0$.
\end{prop}

\begin{proof}
Let $I^0$ be an ideal of $V^0$. By Lemma \ref{IdealCorresLem}, we have  $I^0\subset\idealG(I^0)$.
Thus $I^0\subset V^{0}\cap \idealG(I^0)=\idealR\left(\idealG(I^0)\right)$.
On the other hand, let $u\in \idealR\left(\idealG(I^0)\right)=V^{0}\cap \idealG(I^0)$.
As $u\in V^0$, by Corollary \ref{V0-characterization} again we have
\begin{eqnarray}\label{eu-expre}
u=\sum_{\m\in \Z^{r}} u_{-1,\m}{\bf 1} \  \  \  (\mbox{a finite sum}).
\end{eqnarray}
As $u\in \idealG(I^0)$, by definition $u$ is a linear combination of vectors  of the form $a_{n_{0},\n}v$
with $a\in I^0,\ v\in V,\ (n_{0},\n)\in \Z^{r+1}$. Furthermore, we have
\begin{eqnarray*}
	&&Y\left(a_{n_{0},\n}v;y_0,\y\right)\vac\\
	&=&\Res_{x_0}(x_0-y_0)^{n_{0}}\degr{y}{n}\Yt{a}{x_0}{\n}\Yt{v}{y_0}{\y}\vac \\
	&&-\Res_{x_0}(-y_0+x_0)^{n_{0}}\degr{y}{n}\Yt{v}{y_0}{\y}\Yt{a}{x_0}{\n}\vac\\
	&=&\Res_{x_0}(x_0-y_0)^{n_{0}}\degr{y}{n}\Yt{a}{x_0}{\n}\Yt{v}{y_0}{\y}\vac,
\end{eqnarray*}
which lies in $I^0[[y_{0},y_{1}^{\pm 1},\dots,y_{r}^{\pm 1}]]$ because $a\in I^0,
\  Y(v;y_{0},\y){\bf 1}\in V^0[[y_{0},y_{1}^{\pm 1},\dots,y_{r}^{\pm 1}]]$.
It then follows from (\ref{eu-expre}) that  for every $u \in \idealR\left( \idealG(I^0) \right)$, we have $u\in I^0$.
Thus $\idealR\left( \idealG(I^0) \right)\subset I^0$, and hence
$\idealR\left( \idealG(I^0) \right)=I^0$.
\end{proof}

Note that from Proposition \ref{IdealCorresProp},
 $\idealG$ gives rise to a one-to-one correspondence between
the set of ideals of $V^0$ and the set of the equivalence classes of  ideals of $V$
where ideals $I$ and $J$ of $V$ are said to be {\em equivalent} if $\idealR(I)=\idealR(J)$.

Let $V$ be a \tva{}. For an ideal $I$ of $V$, set
\begin{eqnarray}
K(I)=\{ v\in V\ | \  v_{m_0,\m}V\subset I\  \  \mbox{ for all }(m_0,\m)\in \Z\times \Z^r\}.
\end{eqnarray}
It is straightforward to show that $K(I)$ is an ideal containing $I$.

\begin{prop}\label{K-GR-relation} Let $I$ be any ideal of $V$. Then
$$K(\idealG( \idealR(I))=K(I),\  \  \  \  \idealR(I)=\idealR(K(I)).$$
\end{prop}

\begin{proof} From Lemma \ref{IdealCorresLem}, we have
$$\idealG\big( \idealR(I) \big)={\rm span}\{ a_\suball{m}v\ |\ a\in I\cap V^0,\ (m_0,\m)\in \Z\times \Z^r,\ v\in V\}
\subset I.$$
Set $K'=K(\idealG( \idealR(I))$.
As $\idealG\big( \idealR(I) \big)\subset I$, we have $K'\subset K(I)$.
On the other hand, let $a\in K(I)$.  For $\rangeall{m}$, we have
$a_\suball{m}\vac \in I$ and $a_\suball{m}\vac \in V^0$, so that
\begin{eqnarray*}
a_\suball{m}\vac \in I\cap V^0=\idealR(I)\subset \idealG\big( \idealR(I) \big).
\end{eqnarray*}
It then follows from (\ref{V0-sum}) that
$$Y(a;x_0,\x)v=\sum_{\m\in \Z^r}Y(a_{-1,\m}{\bf 1};x_0,\x)v\in \idealG\big( \idealR(I) \big)[[x_0^{\pm 1},\x^{\pm 1}]]$$
for all $v\in V$. Thus  $a\in K'$.
This proves $K(I)\subset K'$, and hence $K(I)=K'$.

As $I\subset K(I)$, we have $\idealR(I)\subset\idealR(K(I))$.
Note that by Lemma \ref{IdealCorresLem} we have
$$\idealR(K(I))=\{ a_\suball{m}\vac\ |\ a\in K(I),\ \rangeall{m}\}.$$
Furthermore, for $a\in K(I)$, $\rangeall{m}$, we have $a_\suball{m}\vac\in I\cap V^0=\idealR(I)$.
Thus $\idealR(K(I))\subset \idealR(I)$. Therefore $\idealR(I)=\idealR(K(I))$.
\end{proof}

For an ideal of $V$, let $\pi$ be the canonical map from $V$ onto $V/I$.
We see that $\pi^{-1}(\ky{V/I})=K(I)$.
As an immediate consequence of Proposition \ref{K-GR-relation} we have:

\begin{coro} Let $I$ and $J$ be ideals of $V$. Then $\idealR(I)=\idealR(J)$ if and only if
$K(I)=K(J)$, or equivalently, the identity map of $V$
induces an isomorphism from $\overline{V/I}$ to $\overline{V/J}$.
\end{coro}

We have the following result on simplicity:

\begin{prop}\label{SimpleCoro}
Let $V$ be a \tva{}.
Then $\overline{V}$ is a simple \tva{} if and only if $V^0$ is a simple \tva{}.
\end{prop}

\begin{proof}
Assume that $V^0$ is a simple \tva{}. Let $I$ be any ideal of $\overline{V}$. Then $I=J/ \ky{V}$ where $J$ is an ideal of $V$,
 containing $\ky{V}$. As $V^0$ is simple, we have either $J\cap V^0=V^0$ or $J\cap V^0=0$.
 If $J\cap V^0=V^0$, we have ${\bf 1}\in V^0\subset J$, which implies $J=V$. In this case, we have $I=\overline{V}$.
 Assume $J\cap V^0=0$.
 For any $u\in J$, we have $u_{m_0,\m}{\bf 1}\in J\cap V^0=0$ for all $(m_0,\m)\in \Z^{r+1}$, that is,
 $u\in \ky{V}$ by Lemma \ref{Kylemma}. Thus $J\subset \ky{V}$. Conseqently, $J=\ky{V}$ and hence $I=0$.
 Therefore, $\overline{V}$ is simple.

Conversely, assume that $\overline{V}$ is a simple \tva{}. Let $I^0$ be any  ideal of $V^0$.
As $\overline{V}$ is simple,
we have either $\idealG\left( I^0 \right)+\ky{V}=V$ or $\idealG\left( I^0 \right)\subset \ky{V}$.
First, consider the case with $\idealG\left( I^0 \right)+\ky{V}=V$.
Then there exist $u\in\idealG\left(I^0\right)$ and $v\in\ky{V}$ such that $u+v=\vac$.
For  any $w\in V$, we have
 $$w=\Yt{\vac}{x_0}{\x}w=Y(u;x_{0},\x)w+Y(v;x_0,\x)w=\Yt{u}{x_0}{\x}w\in \idealG\left(I^0\right).$$
 Thus  $\idealG\left(I^0\right)=V$, which implies $I^0=\idealR(\idealG(I^0))=\idealR(V)=V^0$ by Proposition \ref{IdealCorresProp}.
We now consider the case with $\idealG\left(I^0\right)\subset \ky{V}$.  We have
$$I^0=\idealR(\idealG\left(I^0\right))\subset \idealR(\ky{V})=\ky{V}\cap V^0=0.$$
Thus, $I^0=0$. This proves that $V^0$ is a simple \tva{}.
\end{proof}


\begin{rem}\label{ResFuncLem}
{\em For a \tva{} $V$, denote by $\Mod V$ the category of $V$-modules.
The following are straightforward analogs of classical facts:
 Let $\varphi:V_1\rightarrow V_2$ be a homomorphism of \tva{}s.
 Then any $V_2$-module $W$ is naturally a $V_1$-module, which is denoted by $\funcRes{\varphi}(W)$.
 Moreover, if $\alpha:W\rightarrow W'$ is a  homomorphism of $V_2$-modules,
 then $\alpha$ is also a $V_1$-module homomorphism from $\funcRes{\varphi}(W)$ to $\funcRes{\varphi}(W')$.
  Set $\funcRes{\varphi}(\alpha)=\alpha$.
 Then $\funcRes{\varphi}$ is an exact faithful covariant functor from $\Mod V_2$ to $\Mod V_1$.
 Furthermore, if $\varphi:V_1\rightarrow V_2$, $\psi:V_2\rightarrow V_3$
  are \rtva{} homomorphisms, then  $\funcRes{\varphi}\circ\funcRes{\psi}=\funcRes{\psi\circ\varphi}$.}
\end{rem}

Recall that a subcategory $\mathcal{C}_1$ of a category $\mathcal{C}$
is called a {\em full subcategory} if
$\Hom_{\mathcal{C}}(A,B)=\Hom_{\mathcal{C}_1}(A,B)$  for any objects $A,B$ in $\mathcal{C}_1$.
 Furthermore, a full subcategory $\mathcal{C}_1$ of $\mathcal{C}$
is called a {\em strictly full subcategory}
if whenever an object $A$ in $\mathcal{C}$ is isomorphic to an object of $\mathcal{C}_1$,
$A$ belongs to $\mathcal{C}_1$.

Here,  we have:

\begin{lem}\label{DominantProp}
Let $V_1$ and $V_2$ be \tva{}s and let $\varphi:V_1\rightarrow V_2$ be a \tva{} homomorphism
such that $V_2^0\subset \Image \varphi$.
Then $\funcRes{\varphi}$ is an isomorphism from
$\Mod V_2$ to a strictly full subcategory of $\Mod V_1$.
\end{lem}

\begin{proof}
Let $W_1,W_2$ be $V_2$-modules. Assume that $\alpha:\funcRes{\varphi}(W_1)\rightarrow\funcRes{\varphi}(W_2)$ is a $V_1$-module homomorphism, or equivalently,  an $\Image \varphi$-module homomorphism.
For $v\in V_2$, $w\in W_1,\  \rangeall{m}$, using (\ref{2.15}) we have
\begin{eqnarray*}
&&\alpha(v_\suball{m}w)=\alpha\big( (\eltzHomo{v}{m})_\suball{m}w \big)=\left(\eltzHomo{v}{m}\right)_\suball{m}\alpha(w)=v_\suball{m}\alpha(w)
\end{eqnarray*}
since $\eltzHomo{v}{m}\in V_2^0\subset \Image \varphi$. This proves that $\alpha$ is also a $V_2$-module homomorphism.
Thus $\funcRes{\varphi}$ is full. The proof for strict fullness is classical: Let $W_1$ be a $V_1$-module, $W_2$ a $V_2$-module, and $\alpha:\funcRes{\varphi}(W_2)\rightarrow W_1$ a $V_1$-module isomorphism. For $u\in \ker \varphi$, $\rangeall{m}$, we have
$$u_\suball{m}\alpha(w_2)=\alpha(u_\suball{m}w_2)=\alpha(\varphi(u)_\suball{m}w_2)=0\  \  \  \mbox{ for }w_2\in W_2.$$
As $\alpha(W_2)=W_1$, we get $u_\suball{m}W_1=0$.
Then $W_1$ is naturally an ${\rm Im}\varphi$-module.
By using $\alpha$, one can make $W_1$ a $V_2$-module denoted by $\tilde{W}_{1}$ such that $\funcRes{\varphi}\tilde{W}_1=W_1$.
\end{proof}

Let $V$ be a \tva{} and let $L$ be a quotient \tva{} of $V$.
Then we have the following commutative diagram
\begin{equation*}
\begin{matrix}
V^0& \longrightarrow & V\\
\downarrow & & \downarrow \\
L^0& \longrightarrow & L
\end{matrix}
\end{equation*}
As an immediate consequence of Proposition \ref{prop:Vmod-V0mod} and Lemma \ref{DominantProp}, we have:

\begin{prop}\label{STvaModCateProp}
A $V$-module $W$ is naturally an $L$-module if and only if $W$ is naturally an $L^0$-module.
Furthermore, $\Mod L$ is the intersection of $\Mod V$ and $\Mod L^0$ as strictly full subcategories of $\Mod V^0$.
\end{prop}

Recall from \cite{LTW2} that a {\em vertex Leibniz algebra} is a vector space $V$ equipped with a linear map
$Y(\cdot,x): V\rightarrow {\mathcal{E}}(V)$ such that
the Jacobi identity for vertex algebras holds.
Combining Remark \ref{eVVbar},  Corollary \ref{barV=V}, and Theorem \ref{ZeroModCateProp},
we immediately have:

\begin{coro}
Suppose that $V$ is a \tva{} such that
$$Y(v;x_0,\x){\bf 1}\in V[[x_0]][x_{1}^{\pm 1},\dots,x_{r}^{\pm 1}]\  \  \  \mbox{ for every }v\in V.$$
Then
$$Y(u;x_0,\x)v\in V[[x_{0}^{\pm 1}]][x_{1}^{\pm 1},\dots,x_{r}^{\pm 1}]\  \  \  \mbox{ for any }u,v\in V.$$
For $u,v\in V$, define
$$Y^{0}(u,x_{0})v=\left(Y(u;x_0,\x)v\right)|_{\x=1}. $$
Then $(V,Y^{0})$  is a vertex Leibniz algebra and $\overline{V}=V/K_{Y}(V)$
is a vertex algebra. Furthermore,
the category $\Mod V$  is naturally isomorphic to
 the category of modules for  $\overline{V}$ viewed as a vertex algebra.
\end{coro}



\section{Irreducible modules for  $r$-loop vertex algebras}
In this section, we continue to study the structure of $V^0$ for a general \rtva{} $V$. First, we show that
$V^0$ has a canonical $\Z^r$-grading which makes $V^0$ a vertex $\Z^r$-graded algebra in a certain sense.
Then we prove that the structure of an \rtva{} $V$ with $V^0=V$ exactly amounts to that of a vertex $\Z^r$-graded algebra.
We also study special vertex $\Z$-graded algebras which are $r$-loop vertex algebras and
we classify their irreducible modules in terms of irreducible modules for the corresponding vertex algebras.

First, we study the \rtva{} structure of $V^0$ for a general \rtva{} $V$.

\begin{de}\label{gradedva}
A {\em vertex $\Z^{r}$-graded algebra}
is a vertex algebra $V$ equipped with a $\Z^{r}$-grading $V=\bigoplus_{\m\in\Z^{r}}V_{(\m)}$ such that ${\bf 1}\in V_{(0)}$  and
$$u_{k}v\in V_{(\m+\n)}\  \  \  \mbox{ for }u\in V_{(\m)},\  v\in V_{(\n)},\  \m,\n\in \Z^{r},\ k\in \Z.$$
\end{de}

\begin{prop}\label{prop:RTVA2ZrGradedVA}
Let $V$ be an \rtva{}. Then $V^0=\bigoplus_{\m\in \Z^r}V_{(\m)}$ is a vertex $\Z^r$-graded algebra, where for $\m\in \Z^r$,
\begin{eqnarray}
V_{(\m)}=\{ v_{-1,-\m}{\bf 1}\  |\   v\in V\}.
\end{eqnarray}
\end{prop}

\begin{proof} For $\m\in \Z^r$, define a linear map $\phi_{\m}: V\rightarrow V_{(-\m)}$
by $\phi_{\m}(u)=u_{-1,\m}{\bf 1}$ for $u\in V$.
From Lemma \ref{simple-facts}, we have
$$\phi_{\m}|_{V_{(-\n)}}=\delta_{\m,\n}\   \   \   \mbox{ for }\m,\n\in \Z^r.$$
It follows that $V^0=\bigoplus_{\m\in \Z^r}V_{(\m)}$ as a vector space.
For $u,v\in V$, $\m,\n\in \Z^r,\ m_0\in \Z$, we have
\begin{eqnarray*}
\Res_{x_0}x_0^{m_0}Y^0(u_{-1,\m}{\bf 1},x_0)(v_{-1,\n}{\bf 1})
&=&\Res_{x_0}x_0^{m_0}Y(u_{-1,\m}{\bf 1};x_0,\x)(v_{-1,\n}{\bf 1})|_{\x=1}\\
&=&\Res_{x_0}x_0^{m_0}\x^{\m}Y(u_{-1,\m}{\bf 1};x_0,\x)(v_{-1,\n}{\bf 1})\\
&=&(u_{-1,\m}{\bf 1})_{m_0,\m}(v_{-1,\n}{\bf 1}),
\end{eqnarray*}
which  by (\ref{tempEq2}) is equal to $\left((u_{-1,\m}{\bf 1})_{m_0,\m}v\right)_{-1,\m+\n}{\bf 1}\in V_{(-\m-\n)}$.
This proves that $V^0$ is a vertex $\Z^r$-graded algebra.
\end{proof}

On the other hand, we have the following result which is straightforward to prove:

\begin{lem}\label{va-rtva}
Let  $V=\bigoplus_{\m\in\Z^{r}}V_{(\m)}$ be a vertex $\Z^{r}$-graded algebra.
For $u\in V_{(\m)}$ with $\m\in \Z^{r}$, define
$$Y(u;x_{0},\x)=Y(u,x_{0})\x^{\m}.$$
Then $V$ is an  \rtva{} with $V^0=V$, which we denote by $F(V)$. Furthermore,
if $\alpha: U\rightarrow V$ is a homomorphism of vertex $\Z^r$-graded algebras,  then $\alpha: F(U)\rightarrow F(V)$
is also a homomorphism of  $(r+1)$-toroidal vertex algebras.
\end{lem}

As an immediate consequence, we have:

\begin{coro}
Let $V$ be a \tva{}.
Denote by $D_0$ the canonical derivation of $V^0$ viewed as a \va{}.
Define linear operators $D_1,\dots,D_r$ on $V^0$ by
\begin{eqnarray}
D_j(v)=m_j v \ \ \ \ \mbox{for }v\in V_{(\m)},\ \m=(m_1,\dots,m_r)\in \Z^r.
\end{eqnarray}
Then $V^0$ equipped with derivations $D_0,D_1,\dots,D_r$ is an extended \tva{}, where $D_1,\dots,D_r$ act on $V^0$ semi-simply with only integer eigenvalues.
Furthermore, the extended \tva{} structure on $V^0$ is unique.
\end{coro}

\begin{de}\label{subcategory0}
Denote by ${\mathcal{C}}_{r+1}^{0}$ the subcategory of $(r+1)$-toroidal vertex algebras, which consists of $V$
satisfying the condition that $V^{0}=V$.
\end{de}

In view of Lemma \ref{va-rtva}, we have a functor $F$ from the category of
$\Z^r$-graded vertex algebras to ${\mathcal{C}}_{r+1}^{0}$. It is clear that $F$ is an isomorphism.
To summarize, we have:

\begin{prop}\label{prop:ZrGradedVAtoETVA}
The category ${\mathcal{C}}_{r+1}^{0}$ is a strictly full subcategory of the category of \etva{}s and
$F$  is an isomorphism from the category of vertex $\Z^r$-graded algebras  to ${\mathcal{C}}_{r+1}^{0}$.
\end{prop}

From Corollary 2.7 and Proposition 2.16 in \cite{LTW} (including the skew-symmetry), we immediately have:


\begin{coro}\label{IdealCriLem}
Let $V$ be a \tva{}  and let $I$ be a left ideal of $V^0$.
Then $I$ is an ideal of $V^0$ (viewed as a \tva{}) if and only if $I$ is stable under the actions of $\der_0,\dots,\der_r$.
\end{coro}

Next, we give a canonical functor from the category of \va{}s to the category of \tva{}s.
Set
\begin{eqnarray}
\Lr=\C[\varr{t}],
\end{eqnarray}
a commutative and associative algebra over $\C$.
Note that any commutative and associative algebra $A$ with identity $1$ is naturally a vertex algebra with
$Y(a,x)b=ab$ for $a,b\in A$ and with $1$ as the vacuum.
In particular, $\Lr$ is naturally a \va{}.

\begin{rem}  {\em Let $A$ be a commutative and associative algebra with identity $1$.
If $W$ is an $A$-module, it can be readily seen that $W$ is a module for $A$ viewed as a vertex algebra. On the other hand,
let $(W,Y_{W})$ be a module for $A$ viewed as a vertex algebra.  For any $a\in A$, we have (cf. \cite{LL})
$$\frac{d}{dx}Y_{W}(a,x)=Y_{W}(D a,x)=0,$$
as $D a=\left(\frac{d}{dx}Y(a,x)1\right)|_{x=0}=0$. Then it is straightforward to show that $W$ is a module
for $A$ viewed as an associative algebra.
Therefore, a module for $A$ viewed as an associative algebra is the same as a module for $A$ viewed as a vertex algebra.}
\end{rem}

For any \vs{} $W$, we set
$$\lr{W}=W\otimes \Lr.$$
Note that $L_r$ is naturally an associative $\Z^r$-graded algebra and a vertex $\Z^r$-graded algebra.
For any vertex algebra $V$, the tensor product vertex algebra $V\otimes L_r$
is naturally a vertex $\Z^r$-graded algebra. In view of Lemma \ref{va-rtva},  $L_r(V)$ is  an \etva{} with
\begin{eqnarray}
\Yt{v\otimes \bdt^\m}{x_0}{\x}= (Y(v,x_0)\otimes \bdt^\m )\degr{x}{m}
\end{eqnarray}
and with derivations $D_0=\der\otimes 1$ and $D_i=-1\otimes \degder{t}{i}$ for $1\le i\le r$.
Furthermore, for any \va{} homomorphism $\varphi:V_1\rightarrow V_2$,
$\varphi\otimes 1$ is an \rtva{} homomorphism from $\lr{V_1}$ to $\lr{V_2}$.
In this way, we have a covariant functor from the category of \va{}s to the category of vertex $\Z^r$-graded algebras,
and then to the category of \rtva{}s.

We have the following technical result:

\begin{lem}\label{LrLem}
Let $V$ be a vertex $\Z^r$-graded algebra and
 let $\varphi:V\rightarrow A$ be a homomorphism of \va{}s.
Define a linear map $\tilde{\varphi}:\  V\rightarrow \lr{A}$ by
\begin{eqnarray*}
\tilde{\varphi}(v)=\varphi(v)\otimes \bdt^{-\m}
\  \   \   \mbox{ for }v\in V_{(\m)} \mbox{ with }\m\in \Z^r.
\end{eqnarray*}
Then $\tilde{\varphi}$ is a homomorphism of vertex $\Z^r$-graded algebras.
Furthermore, if $V$ is graded simple and if $\varphi$ is not zero, then
$\tilde{\varphi}$ is injective.
\end{lem}

\begin{proof} Let  $u\in V_{(\m)},\ v\in V_{(\n)},\ k\in \Z$ with $\m,\n\in \Z^r$.
As $V$ is a vertex $\Z^r$-graded algebra, we have $u_{k}v\in V_{(\m+\n)}$.
Then
$$\tilde{\varphi}(u_{k}v)=\varphi(u_{k}v)\otimes \bdt^{-\m-\n}=\varphi(u)_{k}\varphi(v)\otimes \bdt^{-\m-\n}
=\left(\varphi(u)\otimes \bdt^{-\m}\right)_{k}\left(\varphi(v)\otimes \bdt^{-\n}\right)
=\tilde{\varphi}(u)_{k}\tilde{\varphi}(v).$$
Also, as ${\bf 1}\in V_{({\bf 0})}$, we have $\tilde{\varphi}({\bf 1})={\bf 1}\otimes 1$. Thus
$\tilde{\varphi}$ is a homomorphism of vertex $\Z^r$-graded algebras.
The last assertion is clear.
\end{proof}


\begin{rem}
{\em Let $V$ be a \va{}. We have the tensor product \va{} $L_{r}(V)$.
Note that an $L_{r}(V)$-module structure on a vector space $W$ amounts to a $V$-module structure $Y_{W}(\cdot,x)$
together with an $L_{r}$-module structure such that
$$aY_{W}(v,x)w=Y_{W}(v,x)(aw)\  \  \  \mbox{ for }a\in L_r,\ v\in V,\ w\in W.$$ }
\end{rem}

We have the following straightforward analogue of the notion of evaluation module for affine Lie algebras:

\begin{lem}\label{IrrModDefLem}
Let $V$ be a \va{}, let $(W,Y_W)$ be a $V$-module, and let $\bda\in (\C^\times)^r$.
For $v\in V,\  \m\in \Z^{r}$, set
\begin{eqnarray}
\widehat{Y_{W}}(v\otimes \bdt^\m,x)=\bda^\m Y_{W}(v,x).
\end{eqnarray}
Then $(W,\widehat{Y_W})$ is an $\lr{V}$-module, which is denoted by $W_{\bda}$.
Furthermore, if $W$ is an irreducible $V$-module, then $W_{\bda}$ is an irreducible $\lr{V}$-module.
\end{lem}

Furthermore, we have:

\begin{prop}\label{IrrTensorProp}
Let $V$ be a \va{} of countable dimension (over $\C$) and let $W$ be an irreducible $\lr{V}$-module.
Then $W$  as an $\lr{V}$-module is isomorphic to $U_\bda$ for some irreducible $V$-module $U$
and for some $\bda\in (\C^\times)^r$.
\end{prop}

\begin{proof} Pick a nonzero vector $w$. Since $W$ is an irreducible $\lr{V}$-module,
 it follows (see \cite{Li4}, \cite{DM}) that $W={\rm span}\{ v_{n}w\ |\ v\in \lr{V},\ n\in \Z\}$.
 As $\lr{V}$ is of countable dimension, $W$ is of countable dimension.
 Notice that $L_r$ lies in the center of $L_r(V)$ viewed as a vertex algebra and that
 a module for $\Lr$ viewed as a vertex algebra is the same as
 a module for $\Lr$ viewed as an associative algebra.
Then by the generalized Schur lemma, every element of $L_r$ acts as a scaler on $W$, so that $W$ is necessarily
an irreducible $V$-module. Denote this irreducible $V$-module by $U$.
Then we conclude that $W\simeq U_{\bda}$ as an $L_r(V)$-module
for some $\bda\in (\C^\times)^r$.
\end{proof}

\section{Toroidal vertex algebras associated to toroidal Lie algebras}

In this section, we study toroidal vertex algebras associated to the $r$-loop algebras of affine Kac-Moody Lie algebras
and we study their simple
quotient toroidal vertex algebras.

Let $\g$ be a finite dimensional simple Lie algebra and let $\<\cdot,\cdot\>$ be the killing form suitably normalized.
Let $\hat{\g}=\g\otimes \C[t_0^{\pm 1}]\oplus \C \cent$ be the affine Lie algebra,
where  $\cent$ is the standard central element
and
\begin{equation}
[a\otimes t_0^m,b\otimes t_0^n]=[a,b]\otimes t_0^{m+n}+m\<a,b\>\delta_{m+n,0}\cent
\end{equation}
for $a,b\in \g,\  m,n\in \Z$.
Set
\begin{eqnarray}
\Toro=\hat{\g}\otimes \Lr=\hat{\g}\otimes \C[t_{1}^{\pm 1},\dots,t_{r}^{\pm 1}],
\end{eqnarray}
 where
\begin{equation}\label{LieBracketEq}
[u\otimes \bdt^\m,v\otimes \bdt^\n] = [u,v]\otimes \bdt^{\m+\n}
\end{equation}
for $u,v\in \hat{\g},\  \m, \n\in \Zr$. Note that $L_r(\C\cent)$ $(=\C\cent\otimes L_r)$ lies in the center of $\Toro$.

For $a\in \g$, we set
\begin{equation}
\eltAbs{a}{x_0}{x}=\sum\limits_{\rangeall{m}}\eltAbs{a}{m_0}{m}\degall{x}{m},
\end{equation}
where $\eltAbs{a}{m_0}{m}=a\otimes t_0^{m_0}\bdt^\m$.
We also set
\begin{equation}
\cent(\x)=\sum\limits_{\ranger{m}}\cent(\m)\degr{x}{m},
\end{equation}
where $\cent(\m)=\cent\otimes \bdt^\m$.
Then the Lie bracket relations (\ref{LieBracketEq}) amount to
\begin{eqnarray}\label{GenFuncEq}
[\eltAbs{a}{x_0}{x},\eltAbs{b}{y_0}{y}] &=& \eltAbs{[a,b]}{y_0}{y}\dfunc{x_0}{y_0}{x_0} \delta\left( \frac{\rsymbol{y}}{\rsymbol{x}} \right) \nonumber\\
&&+\<a,b\>\cent(\y)\pder{y_0}\dfunc{x_0}{y_0}{x_0}  \delta\left( \frac{\rsymbol{y}}{\rsymbol{x}} \right),
\end{eqnarray}
where $\delta\left(\frac{\rsymbol{y}}{\rsymbol{x}} \right)=\delta\left(\frac{y_1}{x_1}\right)\cdots \delta\left(\frac{y_r}{x_r}\right)$.

As we need, we briefly recall  from \cite{LTW} a conceptual construction.
Let $W$ be a vector space. An ordered pair $\big(\eltAbs{a}{x_0}{x},\eltAbs{b}{x_0}{x}\big)$
in $\Espaceall{W}$ is said to be {\em compatible} if
there exists non-negative integer $k$ such that
\begin{eqnarray}\label{ComEq}
(x_0-y_0)^k\eltAbs{a}{x_0}{x}\eltAbs{b}{y_0}{y}\in \Hom \left( W, W[[\varr{x},\varr{y}]]((x_0,y_0)) \right).
\end{eqnarray}
A subset $U$ of $\Espaceall{W}$ is said to be {\em local} if for any $\eltAbs{a}{x_0}{x},\eltAbs{b}{x_0}{x}\in U$,
there exists non-negative integer $k$ such that
\begin{eqnarray}\label{localEq}
(x_0-y_0)^k\eltAbs{a}{x_0}{x}\eltAbs{b}{y_0}{y}=(x_0-y_0)^k\eltAbs{b}{y_0}{y}\eltAbs{a}{x_0}{x}.
\end{eqnarray}
Notice that this commutation relation implies (\ref{ComEq}).

\begin{de}
Let $\eltAbs{a}{x_0}{x},\eltAbs{b}{x_0}{x}\in \Espaceall{W}$. Assume $(\eltAbs{a}{x_0}{x},\eltAbs{b}{x_0}{x})$ is compatible. Define
\begin{eqnarray*}
\eltAbs{a}{x_0}{x}_\suball{m}\eltAbs{b}{x_0}{x}\in \Espaceall{W}&&  \fortext\rangeall{m}
\end{eqnarray*}
in terms of generating function
\begin{equation}
\Ye{\eltAbs{a}{y_0}{y}}{z_0}{\z}\eltAbs{b}{y_0}{y}=\sum\limits_{\rangeall{m}}\eltAbs{a}{y_0}{y}_\suball{m}\eltAbs{b}{y_0}{y}\degall{z}{m}
\end{equation}
by
\begin{equation}
\Ye{\eltAbs{a}{y_0}{y}}{z_0}{\z}\eltAbs{b}{y_0}{y}
=z_0^{-k}\left( (x_0-y_0)^k\eltAbs{a}{x_0}{zy}\eltAbs{b}{y_0}{y} \right)\big|_{x_0=y_0+z_0},
\end{equation}
where $k$ is any non-negative integer such that (\ref{ComEq}) holds.
\end{de}

A $\Toro$-module $W$ is called a {\em restricted module} if
\begin{equation}
\eltAbs{a}{x_0}{x}\in \Espaceall{W} \   \   \   \mbox{for all }a\in\g.
\end{equation}

We have:

\begin{lem}\label{ConstructLem}
Let $W$ be a restricted $\Toro$-module. Set
\begin{equation}
U_W={\rm span}\{ 1_{W},\  \cent(\x),\   \eltAbs{a}{x_0}{x}\,|\,a\in \g \}.
\end{equation}
Then $U_W$ is a local subspace of $\Espaceall{W}$ and generates an \rtva{} $\<U_W\>$ with $W$ as a faithful module.
Moreover, $\<U_W\>$ is a $\Toro$-module with
$\eltAbs{a}{y_0}{y}$ acting as  $\Ye{\eltAbs{a}{x_0}{x}}{y_0}{\y}$  for $a\in \g$ and with
$\cent(\y)$ acting as $\Ye{\cent(\x)}{y_0}{\y}$.
\end{lem}

\begin{proof} It follows from (\ref{GenFuncEq}) that $U_W$ is a local subspace of $\Espaceall{W}$.
From \cite{LTW} (Theorem 3.4), $U_{W}$ generates an \rtva{}
$(\<U_W\>,Y_{\mathcal{E}}, 1_W)$ and $W$ is a faithful $\<U_W\>$-module.
The last assertion follows from the commutator relation (\ref{GenFuncEq}) and  the Proposition 3.9 of \cite{LTW}.
\end{proof}

We shall need various subalgebras of $\tau$. Set
\begin{eqnarray}
\borel=\g\otimes \C[t_0], \  \ \nil_+ = \g\otimes t_0\C[t_0], \   \  \nil_- = \g\otimes t_0^{-1}\C[t_0^{-1}].
\end{eqnarray}
Especially, we set
\begin{eqnarray}
\Sc=S(L_r(\C \cent))=S\left(\mathop{\bigoplus}\limits_{\ranger{m}} \C(\cent\otimes \bdt^\m)\right)
\end{eqnarray}
(the symmetric algebra).

We next introduce a special  $L_r(\borel)$-module.

\begin{lem}\label{econstruction-T}
Set $$T=\g\oplus\C\cent\oplus\C,$$
a vector space.
Then $T$ is an $L_{r}(\borel)$-module with
\begin{equation}\label{TopDefEq}
\renewcommand{\arraystretch}{1.5}
\begin{array}{l}
(\eltLie{a}{k}{m})\cdot b=\left\{
	\begin{array}{ll}
	[a,b] \  &\text{if } k=0\\
	\<a,b\>\cent \  & \text{if } k=1 \\
	0 \ & \text{if } k\geq 2,
	\end{array}
\right. \\
(\eltLie{a}{k}{m})\cdot(\C\cent+ \C) =0
\end{array}
\end{equation}
for $a,b\in \g$, $k\geq 0$ and $\ranger{m}$.
\end{lem}

\begin{proof}  It is straightforward  to show (cf. \cite{LTW})  that
$T$ ($=\g+\C\cent+\C$) is a module for $\borel\  \left(=\g\otimes \C[t_0]\right)$ with
\begin{eqnarray*}
(a\otimes t_0^{k})\cdot(b+\alpha \cent+\beta)=\begin{cases}[a,b]&\text{if } k=0\\
\<a,b\>\cent & \text{if } k=1 \\
	0 & \text{if } k\geq 2
\end{cases}
\end{eqnarray*}
for $a,b\in \g$, $\alpha,\beta \in \C$, $k\geq 0$.
Then
$T$ becomes an $L_{r}(\borel)$-module through evaluation $t_i=1$ for $i=1,\dots,r$. That is,
\begin{eqnarray}\label{edifferent-way}
(\eltLie{a}{k}{m})\cdot(b+\alpha \cent+\beta)=\begin{cases}[a,b]&\text{if } k=0\\
\<a,b\>\cent & \text{if } k=1 \\
	0 & \text{if } k\geq 2
\end{cases}
\end{eqnarray}
for $a,b\in \g$, $\alpha,\beta \in \C$, $k\geq 0$ and $\ranger{m}$, as  desired.
\end{proof}

From the definition of $T$, $\g+\C\cent$ is an $L_r(\borel)$-submodule, and we have
\begin{eqnarray}
T=(\g+\C\cent )\oplus \C,
\end{eqnarray}
as an $L_{r}(\borel)$-module. We shall always view $\g+\C\cent$ and $\C$ as $L_r(\borel)$-submodules of $T$.

Form an induced $\tau$-module
\begin{equation}
V(T,0)=U\big(\Toro\big)\otimes_{U(L_{r}(\borel))}T.
\end{equation}
View $T$ as a subspace of $V(T,0)$, so that $\g+\C\cent$ and $\C$ are subspaces of $V(T,0)$.
We see that
$$V(T,0)=U(\tau)T=U(\tau)(\g+\C\cent+\C).$$

Set
$$\vac=1\otimes 1\in V(T,0).$$
As the first main result of this section, we have:

\begin{thm}\label{TvaDefThm}
There exists an \rtva{} structure on $V(T,0)$, which is uniquely determined by the condition that $\vac$
is the vacuum vector and
\begin{eqnarray*}
 \Yt{\cent}{x_0}{\x}=\cent(\x),\   \   \   \Yt{a}{x_0}{\x}=\eltAbs{a}{x_0}{x} \  \  \  \fortext a\in \g.
\end{eqnarray*}
\end{thm}

\begin{proof}
The uniqueness is clear as $V(T,0)$ as a $\tau$-module is generated by $T$. Now, we establish the existence.
Let $W$ be any restricted $\tau$-module. Recall from Lemma \ref{ConstructLem} the local subspace $U_{W}$
of $\Espaceall{W}$ and the $(r+1)$-toroidal vertex algebra  $\<U_{W}\>$ which is naturally a restricted $\tau$-module.
We next show that $U_{W}$ is an $L_{r}(\borel)$-submodule of $\<U_{W}\>$.
For $a,b\in \g$, $j\geq 0$, $\ranger{m}$, from the Proposition 3.9 of \cite{LTW} we have
\begin{equation*}
\renewcommand{\arraystretch}{1.5}
\begin{array}{l}
\eltAbs{a}{x_0}{x}_{j,\m}\eltAbs{b}{x_0}{x}=
\left\{
{
\begin{array}{ll}
\eltAbs{[a,b]}{x_0}{x} & \text{if } j=0, \\
\<a,b\>\cent(\x) & \text{if } j=1, \\
0 & \text{if } j\geq 2,
\end{array}
}\right. \\
\eltAbs{a}{x_0}{x}_{j,\m}1_W=0, \\
\eltAbs{a}{x_0}{x}_{j,\m}\cent(\x)=0.
\end{array}
\end{equation*}
It follows that there exists an $L_{r}(\borel)$-module homomorphism $\phi:T\rightarrow U_W$ such that
$$\phi(a)=\eltAbs{a}{x_0}{x},\ \phi(\cent)=\cent(\x),\  \phi(1)=1_{W}.$$
Then $\phi$ induces a $\Toro$-module homomorphism from $V(T,0)$ into $\<U_W\>$,
which is denoted by $\vmap{W}{x}$.
For $v\in V(T,0)$, define
\begin{equation}
\Yt{v}{x_0}{\x}=\vmap{\T{T}}{x}(v).
\end{equation}
Take $W=V(T,0)$ and set $\vendo{x}=\vmap{\T{T}}{x}$.
Then for $a\in \g$ and $v\in V(T,0)$, we have
\begin{eqnarray*}
&&\Yt{\Yt{a}{z_0}{\z}v}{y_0}{\y} \\
&=&\vendo{y}\left(\eltAbs{a}{z_0}{z}v\right) \\
&=&\Ye{a}{y_0}{\y}\vendo{y}(v) \\
&=&\Res_{x_0}\dfunc{z_0}{x_0-y_0}{z_0}\eltAbs{a}{x_0}{zy}\vendo{y}(v)
		-\dfunc{z_0}{y_0-x_0}{-z_0}\vendo{y}(v)\eltAbs{a}{x_0}{zy} \\
&=&\Res_{x_0}\dfunc{z_0}{x_0-y_0}{z_0}\Yt{a}{x_0}{\z\y}\Yt{v}{y_0}{\y} \\
&&\quad -\Res_{x_0}\dfunc{z_0}{y_0-x_0}{-z_0}\Yt{v}{y_0}{\y}\Yt{a}{x_0}{\z\y}.
\end{eqnarray*}
Now it follows immediately from \cite{LTW} (Theorem 3.10) that $\left( V(T,0),Y,\vac \right)$ carries the structure of an
\rtva{}, as desired.
\end{proof}

Furthermore, we have:

\begin{thm}\label{TvaModThm}
Let $W$ be a restricted $\tau$-module. Then there exists a $V(T,0)$-module structure on $W$,
which is uniquely determined by
$$\Ytmod{W}{\cent}{x_0}{\x}=\cent(\x), \  \  \
\Ytmod{W}{a}{x_0}{\x}=\eltAbs{a}{x_0}{x}\  \  \    \fortext a\in \g.$$
On the other hand, let $(W,Y_W)$ be a $V(T,0)$-module. Then $W$ becomes a restricted $\Toro$-module with
$$\cent(\x)=\Ytmod{W}{\cent}{x_0}{\x},\   \   \
\eltAbs{a}{x_0}{x} =\Ytmod{W}{a}{x_0}{\x} \  \  \fortext a\in \g.$$
Furthermore, the category of restricted $\Toro$-modules is isomorphic to $\Mod V(T,0)$.
\end{thm}

\begin{proof}
We only need to prove the first two assertions.
For the first assertion, the uniqueness is clear since $V(T,0)$ as an  \rtva{} is generated by $T$.
We now establish the existence.
For $a\in \g$, $v\in V(T,0)$, we have
\begin{eqnarray*}
&&\vmap{W}{x}\left( \Yt{a}{y_0}{\y}v \right)=\vmap{W}{x}\big( \eltAbs{a}{y_0}{y}v \big)\\
&&\quad	=\Ye{\eltAbs{a}{x_0}{x}}{y_0}{\y}\vmap{W}{x}(v)=\Ye{\vmap{W}{x}(a)}{y_0}{\y}\vmap{W}{x}(v).
\end{eqnarray*}
By the Lemma 2.10 of \cite{LTW}, $\Ytmod{W}{\cdot}{x_0}{\x}$ is an \rtva{} module homomorphism
from $V(T,0)$ to $\<U_W\>$.
Since $W$ is a canonical $\<U_W\>$-module, $W$ becomes a $V(T,0)$-module with
\begin{equation*}
\Ytmod{W}{v}{x_0}{\x}=\vmap{W}{x}(v)\  \  \  \mbox{ for }v\in V(T,0).
\end{equation*}
For the second assertion, combining (\ref{TopDefEq}) and the commutator formula (2.5) in \cite{LTW}, we get
\begin{eqnarray*}
&&\big[\Ytmod{W}{a}{x_0}{\x},\Ytmod{W}{b}{y_0}{\y}\big] \\
&=&\left(\Ytmod{W}{[a,b]}{y_0}{\y}\dfunc{x_0}{y_0}{x_0}
+ \<a,b\>\Ytmod{W}{\cent}{y_0}{\y}\pder{y_0}\dfunc{x_0}{y_0}{x_0} \right)\delta\left( \frac{\y}{\x} \right)
\end{eqnarray*}
for $a,b\in \g$. It follows that $W$ is a restricted $\Toro$-module with $\eltAbs{a}{x_0}{x} =\Ytmod{W}{a}{x_0}{\x}$ for $a\in \g$
and $\cent(\x)=\Ytmod{W}{\cent}{x_0}{\x}$.
\end{proof}

Recall that $\C$ is an $L_{r}(\borel)$-submodule of $T$.
Set
\begin{eqnarray}
\T{\Sc}=U\left(\Toro\right)\otimes_{U(L_{r}(\borel))}\C,
\end{eqnarray}
which is naturally a $\tau$-submodule of $V(T,0)$.

\begin{lem}\label{VS=VT0}
The $\tau$-submodule $V(\Sc,0)$ of $V(T,0)$ is an $(r+1)$-toroidal vertex subalgebra and we have
$V(\Sc,0)=V(T,0)^0$.
\end{lem}

\begin{proof}
 Recall that ${\bf 1}=1\otimes 1\in T\subset V(T,0)$. We have $V(\Sc,0)=U(\tau){\bf 1}$.
It follows that $V(\Sc,0)$ is the $(r+1)$-toroidal vertex subalgebra generated by the subspace $\g+\C\cent+\C {\bf 1}$
$(=T)$ of $V(T,0)$. On the one hand, as $V(\Sc,0)=U(\tau){\bf 1}$, we have $V(\Sc,0)\subset V(T,0)^0$.
On the other hand, since $V(T,0)=U(\tau)(\g+\C\cent+\C {\bf 1})$, it follows from the Jacobi identity of $V(T,0)$ that $V(T,0)^0\subset U(\tau){\bf 1}=V(\Sc,0)$.
\end{proof}

Combining Lemma \ref{VS=VT0} and Proposition \ref{prop:ZrGradedVAtoETVA}, we immediately have:

\begin{coro}
$V(\Sc,0)$ is an extended $(r+1)$-toroidal vertex algebra such that $V(\Sc,0)^0=V(\Sc,0)$, where the derivations $D_0,D_1,\dots,D_r$ are given by
\begin{eqnarray*}
&&D_0(v)=\left(\frac{\partial}{\partial x_{0}}Y(v;x_0,\x){\bf 1}\right)|_{x_0=0,\x=1},\\
&&D_j(v)=\left(x_j\frac{\partial}{\partial x_{j}}Y(v;x_0,\x){\bf 1}\right)|_{x_0=0,\x=1}
\end{eqnarray*}
for $v\in V(\Sc,0),\ 1\le j\le r$. On the other hand,
 $(V(\Sc,0),{\bf 1},Y^0)$ is a vertex algebra with $D_0$ as its $D$-operator, where
$$Y^{0}(v,x_0)=Y(v;x_{0},\x)|_{\x=1}\  \  \  \mbox{ for }v\in V(\Sc,0).$$
\end{coro}

Noticing that $Y(a;x_0,\x)=a(x_0,\x)$ for $a\in \g$ and $Y(\cent;x_0,\x)=\cent(\x)$ on $V(\Sc,0)$, we immediately have:

\begin{coro}\label{d-bracket-properties}
On $V(\Sc,0)$, the following relations hold  for $a\in \g$ and $1\leq j\leq r$:
\begin{eqnarray*}
&&\big[ \der_0,\eltAbs{a}{x_0}{x} \big] = \pder{x_0}\eltAbs{a}{x_0}{x}=-\pder{t_0}\eltAbs{a}{x_0}{x}, \\
&&\big[ \der_j,\eltAbs{a}{x_0}{x} \big] = \degder{x}{j}\eltAbs{a}{x_0}{x}=-\degder{t}{j}\eltAbs{a}{x_0}{x}, \\
&&\big[ \der_0,\cent(\x) \big]=\pder{x_0}\cent(\x)=0, \\
&&\big[ \der_j,\cent(\x) \big]=\degder{x}{j}\cent(\x)=-\degder{t}{j}\cent(\x).
\end{eqnarray*}
\end{coro}

We define a $\Z\times \Z^r$-grading on $\tau$ by
\begin{eqnarray}\label{degree-def}
\deg (a\otimes t_0^{m_0}\bdt^{\m})=-(m_0,\m),\   \   \   \  \deg (\cent \otimes \bdt^{\m})=-(0,\m)
\end{eqnarray}
for $a\in \g,\ (m_0,\m)\in \Z\times \Z^r$. This makes $\tau$ a $\Z\times \Z^r$-graded Lie algebra.
Note that $L_r(\g\otimes \C[t_0])$ ($=L_r(\borel)$) is a graded subalgebra. By assigning $\deg {\bf 1}=(0,{\bf 0})$,
$V(\Sc,0)$ becomes  a $\Z\times \Z^r$-graded $\tau$-module
\begin{equation}
V(\Sc,0)=\mathop{\bigoplus}\limits_{(m_0,\m)\in \Z\times \Z^r}V(\Sc,0)_{(m_0,\m)}.
\end{equation}
For $\m\in \Z^r$, set
\begin{eqnarray}
V(\Sc,0)_{(\m)}=\bigoplus_{m_0\in \Z}V(\Sc,0)_{(m_0,\m)}.
\end{eqnarray}
On the other hand, for $n\in \Z$, set
\begin{equation}
V(\Sc,0)_{(n)}=\mathop{\bigoplus}\limits_{\ranger{m}}V(\Sc,0)_{(n,\m)}.
\end{equation}

We have:

\begin{lem}
Equipped with the $\Z$-grading $V(\Sc,0)=\oplus_{n\in \Z}V(\Sc,0)_{(n)}$,
vertex algebra $V(\Sc,0)$  becomes a $\Z$-graded vertex algebra in the sense that ${\bf 1}\in V(\Sc,0)_{(0)}$ and
\begin{eqnarray}\label{ezgva-relation}
u_{k}v\in V(\Sc,0)_{(m+n-k-1)}
\end{eqnarray}
for $u\in V(\Sc,0)_{(m)},\ v\in V(\Sc,0)_{(n)},\ m,n,k\in \Z$.
Furthermore, we have $V(\Sc,0)_{(n)}=0$ for $n<0$,
\begin{eqnarray}
 V(\Sc,0)_{(0)}=\Sc, \mbox{ and }\  V(\Sc,0)_{(1)}=L_r(\g\otimes t_0^{-1})\otimes \Sc.
\end{eqnarray}
\end{lem}

\begin{proof}  It is clear that ${\bf 1}\in V(\Sc,0)_{(0)}$. Note that for $a\in \g,\ \m\in \Z^r$, we have
$$Y^0\left((a\otimes t_0^{-1}\bdt^{\m}){\bf 1},x_0\right)
=\left(\sum_{k\in \Z} (a\otimes t_0^{k}\bdt^{\m})x_0^{-k-1}\x^{-\m}\right)_{\x=1}
=\sum_{k\in \Z} (a\otimes t_0^{k}\bdt^{\m})x_0^{-k-1},$$
$$Y^{0}((\cent\otimes \bdt^{\m}){\bf 1},x_0)=\cent\otimes \bdt^{\m}.$$
It then follows that $V(\Sc,0)$ as a vertex algebra is generated by the subspace
$L_{r}(\g\otimes t_0^{-1}){\bf 1}+L_{r}(\cent){\bf 1}$. We see that (\ref{ezgva-relation}) holds
for $u\in L_{r}(\g\otimes t_0^{-1}){\bf 1}$
and for $u\in L_{r}(\cent){\bf 1}$. Then it follows from induction that  (\ref{ezgva-relation}) holds for general $u$.
This proves that $V(\Sc,0)$ is a $\Z$-graded vertex algebra. The furthermore assertion is clear.
\end{proof}

For  $0\leq j\leq r$, denote by $d_j$ the derivation $1\otimes \degder{t}{j}$ on $\Toro$.
Set
\begin{eqnarray}
\ToroD=\Toro\rtimes (\C d_0+\cdots+\C d_r).
\end{eqnarray}

We have:

\begin{lem}\label{tilde-tau-module}
The $(r+1)$-toroidal vertex algebra $V(\Sc,0)$ is a $\ToroD$-module with $d_{1},\dots,d_{r}$ acting as $-D_{1},\dots,-D_{r}$ and with $d_0$ acting as $-L(0)$, where $L(0)$ is the linear operator on $V(\Sc,0)$ defined by
\begin{eqnarray}
L(0)v=nv\  \  \  \mbox{ for }v\in V(\Sc,0)_{(n)}, \ n\in \Z.
\end{eqnarray}
\end{lem}

\begin{proof} For $a\in \g,\ (m_0,\m)\in \Z\times \Z^r$,  as
$$\deg (a\otimes t_0^{m_0}\bdt^{\m})=-(m_0,\m), \   \  \  \   \deg (\cent \otimes \bdt^{\m})=-(0,\m),$$
we have
$$[L(0),(a\otimes t_0^{m_0}\bdt^{\m})]=-m_0 (a\otimes t_0^{m_0}\bdt^{\m}),
\   \  [L(0),(\cent \otimes \bdt^{\m})]=0$$
and $[L(0),D_{j}]=0$ for $1\le j\le r$,
as needed.
\end{proof}

Note that $V(T,0)$ is a $\tau$-module and an $\Sc$-module.
 For $u\in \Sc$, denote by $\rho(u)$ the corresponding operator on $V(T,0)$.

\begin{lem}\label{Sc-structure}
The subspace $\Sc$ is a subalgebra of $V(\Sc,0)$ viewed as
an $(r+1)$-toroidal vertex algebra and a vertex algebra, where for $u\in (\Sc)_{\m}$ with $\m\in \Z^r$,
\begin{eqnarray}
Y(u;x_{0},\x)=\rho(u) \x^{\m}\  \  \mbox{ and }\  \   Y^0(u,x_0)=\rho(u).
\end{eqnarray}
Furthermore, every $V(\Sc,0)$-module is naturally an $\Sc$-module.
\end{lem}

\begin{proof} Noticing that $\cent\in T\subset V(T,0)$ and $Y(\cent;x_0,\x)=\cent(\x)$, we have
\begin{eqnarray}\label{cent-components}
\cent_{m_0,\m}=\delta_{m_0,-1}\rho(\cent\otimes {\bf t}^{\m})
\end{eqnarray}
for $m_0\in \Z,\  \m\in \Z^r$, which gives
$$\cent_{m_0,\m}{\bf 1}=\delta_{m_0,-1}(\cent\otimes {\bf t}^{\m})\in \Sc.$$
Using this relation, formula (2.13), and (\ref{cent-components}) we get
\begin{eqnarray}
Y(\cent\otimes \bdt^{\m};x_0,\x)=Y(\cent_{-1,\m}{\bf 1};x_0,\x)=\rho(\cent\otimes \bdt^{\m})\x^{-\m}.
\end{eqnarray}
Then we have
\begin{eqnarray}
Y^0(\cent\otimes \bdt^{\m},x_0)=Y(\cent\otimes \bdt^{\m};x_0,\x)|_{\x=1}=\rho(\cent\otimes \bdt^{\m}).
\end{eqnarray}
It then follows from induction and the Jacobi identity.
\end{proof}

Next, we  study simple \rtvqa{}s of $V(T,0)$.
By Proposition  \ref{SimpleCoro}, simple quotients of $V(T,0)$ correspond to simple \rtvqa{}s of $V(\Sc,0)$.
Recall that every $V(\Sc,0)$-module is naturally a module for $\Sc$ as an associative algebra.
On the other hand,  to an $\Sc$-module we can associate a $V(\Sc,0)$-module as follows:
Let $U$ be an $\Sc$-module. We let $L_r(\borel)$ act on $U$ trivially, making $U$ an $L_r(\borel+\C\cent)$-module.
Then set
\begin{eqnarray}
M(U)=U(\tau)\otimes_{U(L_r(\borel+\C\cent))} U,
\end{eqnarray}
which is a restricted $\tau$-module.
In view of Theorem \ref{TvaModThm}, $M(U)$ is naturally a $V(T,0)$-module
and a $V(\Sc,0)$-module. In particular, we have $V(\Sc,0)=M(\Sc)$.

\begin{rem}
{\em Note that  every irreducible $\tau$-module is of countable dimension over $\C$.
As $\Sc$ lies in the center of $U(\tau)$, each element of $\Sc$ necessarily acts as a scalar  on every irreducible $\tau$-module.
Thus, an algebra homomorphism $\psi: \Sc\rightarrow \C$ is attached to every irreducible $\tau$-module. }
\end{rem}

Now, we establish some basic results.

\begin{lem}\label{IdealCriToroLem}
An ideal of $V(\Sc,0)$ viewed as an \rtva{} (resp, a \va{})
exactly amounts to a $D_0$-stable and $\Z^{r}$-graded  (resp. $D_0$-stable) $\tau$-submodule.
\end{lem}

\begin{proof} As $V(\Sc,0)$ is generated from ${\bf 1}$ by the coefficients of vertex operators
$$Y(a;x_0,\x)\  (=a(x_0,\x))\   \  (a\in \g)\  \mbox{ and }Y(\cent;x_0,\x)\   (=\cent(\x)),$$
 it follows that a left ideal of $V(\Sc,0)$ amounts to a $\tau$-submodule. Then by Corollary \ref{IdealCriLem},
an ideal of $V(\Sc,0)$ viewed as an \rtva{} (resp, a \va{})
amounts to a $D_0$-stable and $\Z^{r}$-graded  (resp. $D_0$-stable) $\tau$-submodule.
\end{proof}

\begin{lem}\label{ideal}
Let $K$ be a $\Z^r$-graded ideal (resp. an ideal) of $\Sc$. Then $U(\tau)K$
 is a $\Z$-graded ideal of $V(\Sc,0)$ viewed as an \rtva{} (resp. a \va{}).
\end{lem}

\begin{proof} Notice that $L(-1)\Sc=0$, which implies
$L(-1)K=0$. It follows that  $L(-1)\left(U(\tau)K\right)\subset U(\tau)K$.
Then $U(\tau)K$  is an ideal of $V(\Sc,0)$ viewed as a \va{}.
If $K$ is $\Z^r$-graded, then $U(\tau)K$  is also $\Z^r$-graded.
By Lemma \ref{IdealCriToroLem}, $U(\tau)K$  is an ideal of $V(\Sc,0)$ viewed as an \rtva{} in this case.
\end{proof}

On the other hand,  we have:

\begin{lem}\label{maximal-ideal}
Let $J$ be a maximal $\Z$-graded ideal of $V(\Sc,0)$ viewed as an \rtva{} (resp. a \va{}).
Then $J_{(0)}$ is a maximal $\Zr$-graded ideal (resp. maximal ideal) of $\Sc$.
\end{lem}

\begin{proof} Let $I$ be any proper $\Z^r$-graded ideal of $\Sc$, containing $J_{(0)}$. By Lemma \ref{ideal},
$U(\tau)I$ is a $\Z\times \Z^{r}$-graded ideal of $V(\Sc,0)$ with $(U(\tau)I)_{(0)}=I$.
Then $J+U(\tau)I$ is a $\Z^r$-graded ideal of $V(\Sc,0)$, containing $J$.
Since $\left(J+U(\tau)I\right)_{(0)}=I$, $J+U(\tau)I$ is a proper $\Z$-graded ideal. It follows that $J+U(\tau)I=J$. Thus
$I= J_{(0)}$. This proves that $J_{(0)}$ is maximal.
\end{proof}

\begin{de}Let $A=\Sc/I$ be a quotient algebra of $\Sc$ (with $I$ an ideal of $\Sc$). Define
\begin{eqnarray}
V(A,0)=V(\Sc,0)/U(\tau)I\simeq U(\tau)\otimes_{U(L_r(\borel+\C\cent))}A,
\end{eqnarray}
which is a vertex algebra by Lemma \ref{ideal}. If $A$ is also $\Z^r$-graded, $V(A,0)$ is an
$(r+1)$-toroidal vertex algebra.
\end{de}

\begin{prop}\label{simple-quotients}
For any $\Z^r$-graded simple quotient algebra $A$ of $\Sc$, $V(A,0)$ has a unique
maximal $\Z\times \Z^r$-graded ideal $J(A)$. Denote  by $L(A,0)$ the $\Z$-graded
simple quotient $(r+1)$-toroidal vertex algebra.
Furthermore, any simple $\Z$-graded quotient of $(r+1)$-toroidal vertex algebra $V(\Sc,0)$ is isomorphic to $L(A,0)$
for some $\Z^r$-graded simple quotient algebra $A$ of $\Sc$.
\end{prop}

\begin{proof} Note that $V(A,0)$ is naturally $\Z\times \Z^{r}$-graded with
$A=V(A,0)_{(0)}.$ Let $J(A)$ be the sum of all $\Z\times \Z^{r}$-graded ideals $P$ of $V(A,0)$ such that $P_{(0)}=0$.
It is straightforward to see that $J(A)$ is the unique maximal $\Z\times \Z^r$-graded ideal.
Now, let $J$ be any maximal $\Z$-graded ideal of $(r+1)$-toroidal vertex algebra $V(\Sc,0)$.
By Lemma \ref{IdealCriToroLem}, $J$ is $\Z\times \Z^{r}$-graded.
Set
$J_{(0)}=\bigoplus_{\m\in \Z^r}J_{(0,\m)}$. By Lemma \ref{maximal-ideal}, $J_{(0)}$ is a maximal $\Z^r$-graded ideal of $\Sc$.
 We have
$\left(V(\Sc,0)/J\right)_{(0)}=\Sc/J_{(0)}$,
$$L_{r}(\borel)(\Sc/J_{(0)})=0, \  \mbox{ and }\  V(\Sc,0)/J=U(\tau)(\Sc/J_{(0)}).$$
Set $A=\Sc/J_{(0)}$. By the construction of $V(A,0)$,  there exists a $\tau$-module homomorphism
from $V(A,0)$ to $V(\Sc,0)/J$, extending the identity map on $A$. It follows that
 $V(\Sc,0)/J\simeq L(A,0)$.
\end{proof}

\begin{rem}\label{rem-rao}
{\em In view of Proposition \ref{simple-quotients}, to determine simple $\Z$-graded quotients of $(r+1)$-toroidal vertex algebra $\T{\Sc}$, it suffices to determine $\Z^r$-graded simple quotients of $\Sc$.
Note that every $\Z^r$-graded simple quotient algebra of $\Sc$ is isomorphic to the image of a $\Z^r$-graded algebra
homomorphism from $\Sc$ to $L_r$ and that a $\Z^r$-graded simple quotient of $\Sc$
 is an irreducible $\Zr$-graded $\Sc$-module.
 It was proved by Rao (\cite{R1}, Lemma 3.3) that every irreducible $\Zr$-graded $\Sc$-module is
isomorphic to $\Image \psi$ for some $\Zr$-graded algebra homomorphism $\psi:\Sc\rightarrow \Lr$.}
\end{rem}

\begin{de}\label{CenterSimpleTvaDef}
Let $\psi:\Sc\rightarrow \Lr$ be a $\Zr$-graded algebra homomorphism.
Define $\T{\psi}$ to be the \rtvqa{} of $V(\Sc,0)$ modulo the ideal generated by $\ker\psi$, which is $U(\tau)(\ker\psi)$.
\end{de}

Alternatively, we have
\begin{eqnarray}
V(\psi,0)=V(A_{\psi},0)=U(\tau)\otimes_{U(L_r(\borel+\C\cent))} A_{\psi},
\end{eqnarray}
where $A_{\psi}=\Sc/\Ker\psi\simeq \Image\psi$.
As $U(\tau)(\ker \psi)$ is a $\Z\times \Z^r$-graded ideal, $V(\psi,0)$ is $\Z$-graded with
\begin{eqnarray}
V(\psi,0)_{(0)}=A_{\psi}\  \  \mbox{ and }\  \  V(\psi,0)_{(n)}=0\   \mbox{ for }n<0.
\end{eqnarray}

In view of Rao's result,
to determine simple $\Z$-graded quotients of $\T{\Sc}$, we need to determine simple quotients of $\T{\psi}$
for all $\Zr$-graded algebra homomorphisms $\psi:\Sc\rightarrow \Lr$ such that
$\Image\psi$ are $\Zr$-graded simple $\Sc$-modules.

\begin{lem}\label{SimpleQuotientLem}
As a $\tau$-module, $V(\psi,0)$  has a unique maximal $\Z\times \Z^r$-graded
$\tau$-submodule, which is denoted by $M_{\psi}$.
Moreover, $M_{\psi}$  is the unique maximal $\Z$-graded ideal of \rtva{} $V(\psi,0)$.
\end{lem}

\begin{proof} Note that $V(\psi,0)_{(0)}=A_{\psi}\simeq {\rm Im}\psi$, which is a $\Z^r$-graded irreducible $\Sc$-module
and that  $V(\psi,0)=U(\tau)A_{\psi}$.
In view of this, for every proper  $\Z\times \Z^r$-graded $\tau$-submodule $W$ of $V(\psi,0)$, we have $W_{(0)}=0$.
Then it follows that the sum  $M_{\psi}$ of all proper $\Z\times \Z^r$-graded $\tau$-submodules of $V(\psi,0)$ is
the (unique) maximal proper $\Z\times \Z^r$-graded $\tau$-submodule. To prove that $M_{\psi}$ is an ideal of $V(\psi,0)$,
by Lemma \ref{IdealCriToroLem}, we need to show
$\der_0(M_{\psi})\subset M_{\psi}$.
Since $D_{0}V(\psi,0)_{(n)}\subset V(\psi,0)_{(n+1)}$ for $n\in \Z$ and
$(M_{\psi})_{(0)}=0$, we have
\begin{eqnarray*}
\C[\der_0]M_{\psi}\subset \mathop{\bigoplus}\limits_{k>0} V(\psi,0)_{(k)}.
\end{eqnarray*}
Note that
$$[D_0,a\otimes t_0^{k}\bdt^{\m}]=-k(a\otimes t_0^{k-1}\bdt^{\m}),\  \  \  \   [D_0,\cent\otimes \bdt^{\m}]=0$$
for $a\in \g,\ k\in \Z,\ \m\in \Z^r$.
It follows that $\C[\der_0]M_{\psi}$ is a proper $\Z\times \Z^r$-graded $\tau$-submodule.
From the definition of $M_{\psi}$, we have $\C[\der_0]M_{\psi}\subset M_{\psi}$, which implies
$\der_0 M_{\psi}\subset M_{\psi}$.
By Lemma \ref{IdealCriToroLem}, $M_{\psi}$ is an ideal.
\end{proof}

\begin{de}
Let $\psi:\Sc\rightarrow \Lr$ be a $\Zr$-graded algebra homomorphism
such that $\Image\psi$ is a $\Zr$-graded simple $\Sc$-module.
Define $\Stva{\psi}$ to be the unique simple $\Z$-graded \rtvqa{} of $\T{\psi}$.
\end{de}

Let $E(1):\Lr\rightarrow \C$ be the evaluation map with $t_j=1$ for $1\le j\le r$.
For any  algebra homomorphism $\psi:\Sc\rightarrow \Lr$, we set
\begin{eqnarray}
\bar{\psi}=E(1)\circ \psi,
\end{eqnarray}
an algebra homomorphism from $\Sc$ to $\C$.
Conversely, any algebra homomorphism arises this way from a $\Zr$-graded algebra homomorphism $\psi$.

\begin{de}\label{V0}
Let  $\psi:\Sc\rightarrow \Lr$ be an algebra homomorphism.  We define
\begin{equation}
\overline{V}(\psi,0)=V(\Sc,0)/U(\tau)(\Ker \bar{\psi})=U(\tau)\otimes_{U(L_{r}(\borel +\C\cent))}\left(\Sc/\ker \bar{\psi}\right),
\end{equation}
a \vqa{}  of $V(\Sc,0)$.
\end{de}

Setting $A_{\bar{\psi}}=\Sc/(\ker \bar{\psi})$, we have
\begin{eqnarray}
\overline{V}(\psi,0)=V(A_{\bar{\psi}},0)=U(\tau)\otimes_{U(L_{r}(\borel +\C\cent))}A_{ \bar{\psi}}.
\end{eqnarray}
Note that $U(\tau)(\ker \bar{\psi})$ is a $\Z$-graded ideal. Thus $\overline{V}(\psi,0)$ is $\Z$-graded with
$$\overline{V}(\psi,0)_{(0)}=A_{\bar{\psi}}\  \  \mbox{ and }\  \  \overline{V}(\psi,0)_{(n)}=0\  \  \mbox{ for }n<0.$$
As $\ker \psi\subset \ker \bar{\psi}$, $\overline{V}(\psi,0)$ is naturally a homomorphism image of $V(\psi,0)$.

Similarly, we have:

\begin{lem}\label{SimpleQuotientVA}
As a $\tau$-module, $\overline{V}(\psi,0)$ has a unique maximal $\Z$-graded $\tau$-submodule denoted by $M_{\bar{\psi}}$,
which is the unique maximal $\Z$-graded ideal of vertex algebra $\overline{V}(\psi,0)$.
\end{lem}

Define
\begin{eqnarray}
\VSzero{\psi}=\overline{V}(\psi,0)/M_{\bar{\psi}},
\end{eqnarray}
a simple $\Z$-graded vertex algebra.

We have:

\begin{coro}
Let $\psi:\Sc\rightarrow \Lr$ be a $\Zr$-graded algebra homomorphism
such that $\Image\psi$ is a $\Zr$-graded simple $\Z$-graded $\Sc$-module.
Then $\VSzero{\psi}$ is isomorphic to a simple $\Z$-graded \vqa{} of $L(\psi,0)$.
\end{coro}

Recall that $L^0(\psi,0)$ is the unique simple $\Z$-graded quotient vertex algebra of $\overline{V}(\psi,0)$
 and that $L^0(\psi,0)$ is a $\Z$-graded irreducible $\tau$-module generated by vector $1_{\psi}$
satisfying the condition:
$$L_r(\borel)1_{\psi}=0,\   \   \   u\cdot 1_{\psi}=\bar{\psi}(u)1_{\psi}\   \   \   \mbox{ for }u\in \Sc.$$

Let $\gamma$ be a linear functional on $L_r$.  We define
\begin{eqnarray}
I_{\gamma}=\{ u\in L_r\ |\ \gamma(uL_r)=0\}.
\end{eqnarray}
It can be readily seen that $I_{\gamma}$ is an ideal of $L_r$ such that $\gamma(I_{\gamma})=0$ and
it is the (unique) maximal ideal with this property.
Define a bilinear form $\<\cdot,\cdot\>_{\gamma}$ on $L_r(\g)$ by
\begin{eqnarray}
\< a\otimes \bdt^{\m},b\otimes \bdt^{\n}\>_{\gamma}=\<a,b\>\gamma(\bdt^{\m+\n})
\end{eqnarray}
for $a,b\in \g,\ \m,\n\in \Z^r$. It is clear that $\<\cdot,\cdot\>_{\gamma}$ is a symmetric invariant bilinear form on $L_r(\g)$.
Furthermore, we have:

\begin{lem}\label{kernel}
The kernel of the bilinear form $\<\cdot,\cdot\>_{\gamma}$ on $L_r(\g)$  is $\g\otimes I_{\gamma}$.
\end{lem}

\begin{proof} It is clear from the definitions that $\g\otimes I_{\gamma}$ is contained in the kernel.
For convenience, let $K$ denote the kernel. Then $K$ is a $\g$-submodule of $L_r(\g)$.
Since $\g$ is a finite-dimensional simple Lie algebra (over $\C$), we have $K=\g\otimes U$
for some subspace $U$ of $L_r$. Choose $a,b\in \g$ such that $\<a,b\>=1$. For any $u\in U,\ v\in L_r$, we have
$$0=\<a\otimes u,b\otimes v\>_{\gamma}=\<a,b\>\gamma(uv)=\gamma(uv).$$
This implies $U\subset I_{\gamma}$, proving $K=\g\otimes U\subset \g\otimes I_{\gamma}$.
\end{proof}

For a linear functional $\gamma: L_r\rightarrow \C$, set
\begin{eqnarray}
G(r,\g,\gamma)=\g\otimes (L_r/I_{\gamma})
\end{eqnarray}
and equip $G(r,\g,\gamma)$ with the symmetric invariant bilinear form $\<\cdot,\cdot\>_{\gamma}$, which is non-degenerate
by Lemma \ref{kernel}. Denote by $\widehat{G}(r,\g,\gamma)$ the affine Lie algebra.
It is known that for any complex number $\ell$, we have a simple $\Z$-graded vertex algebra
$L_{\widehat{G}(r,\g,\gamma)}(\ell,0)$.
Let $\psi:\Sc\rightarrow \Lr$ be a $\Zr$-graded algebra homomorphism.
Define $\gamma_\psi:\Lr\rightarrow \C$ by $\bdt^\m\mapsto \bar{\psi}(\cent\otimes \bdt^\m)$ for $\m\in \Zr$.
Then we have:

\begin{lem}
The simple $\Z$-graded vertex algebra $L^0(\psi,0)$ is isomorphic to $L_{\widehat{G}(r,\g,\gamma_\psi)}(\ell,0)$ with $\ell=\gamma_\psi(1)$.
In particular,  $L^0(\psi,0)_{(1)}$ is isomorphic to $G(r,\g,\gamma_\psi)$.
\end{lem}

\begin{proof} Recall that
$L^0(\psi,0)$ is an $\N$-graded vertex algebra with $L^0(\psi,0)_{(0)}=\C {\bf 1}$ and $L^0(\psi,0)_{(1)}$ is linearly spanned by $a_{-1,\m}{\bf 1}$ for $a\in \g,\ \m\in \Z^r$.
Note that $L^0(\psi,0)_{(1)}$ is a Lie algebra.
Let $\phi: L_r(\g)\rightarrow L^0(\psi,0)_{(1)}$ be the linear map given by
$\phi(a\otimes \bdt^{\m})=a_{-1,\m}{\bf 1}$ for $a\in \g,\ \m\in \Z^r$. This is a Lie algebra homomorphism.
We show that $\ker \phi=\g\otimes I_{\gamma}$. Let $a\in \g,\ u\in I_{\gamma}$. For any $b\in \g,\ v\in L_r$, we have
$$(b\otimes v)_{1}(a\otimes u)_{-1}{\bf 1}=0.$$
Then $U(\tau)(a\otimes u)_{-1}{\bf 1}$ is a proper graded submodule of $L^0(\psi,0)$.
As $L^0(\psi,0)$ is a $\Z$-graded irreducible $\tau$-module, we must have $(a\otimes u)_{-1}{\bf 1}=0$.
This shows that $\g\otimes I_{\gamma}\subset \ker \phi$. Consequently,
 $L^0(\psi,0)_{(1)}$ is isomorphic to $\g\otimes (L_r/I_{\gamma})$.
\end{proof}

Let $s$ be a positive integer, let  $\bda_1,\dots,\bda_s\in (\C^\times)^r$ be distinct,
and let  $\ell_1,\dots,\ell_s$ be complex numbers. To these data, we associate
a $\Zr$-graded algebra homomorphism $\psi:\Sc\rightarrow \Lr$ determined by
\begin{equation}\label{IntCentActEq}
\psi(\cent\otimes \bdt^\m)
=\left(\sum\limits_{i=1}^{s}\ell_i \bda_i^\m\right)\bdt^\m
\end{equation}
for $\m=(m_1,\dots,m_r)\in \Z^r$, where $\bda_i^\m=a_{i1}^{m_1}\cdots a_{ir}^{m_r}$.
We have
\begin{equation}\label{ebarpsi}
\bar{\psi}(\cent\otimes \bdt^\m)=\sum\limits_{i=1}^{s}\ell_i \bda_i^\m.
\end{equation}

\begin{rem}\label{rem:comp123456} {\em
For each $1\leq i\leq r$, suppose $b_{i1},\dots,b_{iN_i}$ are all the distinct elements of
$\{ a_{1i},\dots,a_{si} \}$.
Set
\begin{eqnarray*}
B=\{ (b_{1j_1},\dots,b_{rj_r}) \,|\, 1\leq j_i\leq N_i\ \mbox{for }1\leq i\leq r  \}\subset (\C^{\times})^{r}
\end{eqnarray*}
with $|B|=N_1N_2\cdots N_r$.
From definition we have $\{\bda_1,\dots,\bda_s\}\subset B$.
Write $B=\{\rsymbol{b}_{1},\dots, \rsymbol{b}_{N}\}$ with $N=N_1N_2\cdots N_r$
such that $\rsymbol{b}_i=\bda_i$ ($1\leq i\leq s$),
noticing that $\bda_1,\dots, \bda_s$ are distinct.
Then $\bar{\psi}$ defined in (\ref{ebarpsi})
is a special case of the algebra homomorphism $\psi$ defined in (\cite{R2}; (3.14)).
From \cite{R2} (Lemma 3.11), the Lie algebra homomorphism
$\Phi:\lr{\hat{\g}}\rightarrow \hat{\g}^{\oplus N}$,
which is given by $a\otimes \bdt^\m\mapsto (\rsymbol{b}_i^\m a)_{1\leq i\leq N}$ ($a\in\hat{\g}$), is surjective.
Denote by $\pi$ the projection map from $\hat{\g}^{\oplus N}$ to $\hat{\g}^{\oplus s}$ via
$(a_i)_{1\leq i\leq N}\mapsto (a_i)_{1\leq i\leq s}$. Then $\pi\circ \Phi$ is surjective.}
\end{rem}

Note that for any complex number $\ell$, one has a simple $\Z$-graded vertex algebra $L_{\hat{\g}}(\ell,0)$
associated to the affine Lie algebra $\hat{\g}$.
Furthermore, if $\ell\ne -h^{\vee}$, where $h^{\vee}$ is the dual Coxeter number of $\g$,
$L_{\hat{\g}}(\ell,0)$ is a simple vertex operator algebra.
We have:

\begin{prop}\label{SzeroStructProp}
Let $s$ be a positive integer, let $\ell_1,\dots,\ell_{s}$ be complex numbers, and let $\bda_{1},\dots,\bda_{s}\in (\C^{\times})^{r}$ be distinct. Let $\psi:\Sc\rightarrow L_r$ be defined as in
(\ref{IntCentActEq}).
Then
there exists a \va{} isomorphism which is also a $\Toro$-module isomorphism
\begin{equation}
\eta:\VSzero{\psi}\rightarrow L_{\hat{\g}}(\ell_1,0)\otimes \cdots\otimes L_{\hat{\g}}(\ell_s,0),
\end{equation}
 such that
\begin{eqnarray}
\eta(a_{-1,\m}\vac)=\sum\limits_{i=1}^s \bda_i^\m \left(\vac \otimes \cdots \otimes
		\mathop{a_{-1}\vac}_{\substack{\uparrow \\ i\text{-th}}} \otimes \cdots \otimes \vac\right)
\end{eqnarray}
for $a\in \g$, $\ranger{m}$. Furthermore, if $\ell_{i}\ne -h^{\vee}$ for $1\le i\le s$,
$L^0(\psi,0)$ is a simple vertex operator algebra.
\end{prop}

\begin{proof} From \cite{R2} (Lemma 3.11) (see Remark \ref{rem:comp123456}),
the right hand side is an irreducible $\Z$-graded $\tau$-module
with $\cent \otimes \bdt^{\m}$ acting as scalar
$$\sum_{i=1}^{s}\ell_{i}\bda_{i}^{\m}\  \  \left(=\bar{\psi}(\cent\otimes \bdt^{\m})\right)$$
and with $a\otimes t_0^{k}\bdt^{\m}$ acting as
$$\sum_{i=1}^{s}\ell_i \bda_{i}^{\m}\pi_{i}(a\otimes t_0^{k})$$
for $a\in \g,\ (k,\m)\in \Z\times \Z^r$, where $\pi_{i}$ is the Lie algebra embedding of $\hat{\g}$ into $\hat{\g}^{\oplus s}$.
Let $V$ denote this $\tau$-module locally and let ${\bf 1}_{R}$ denote the generator ${\bf 1}^{\otimes s}$.
We have
$$L_r(\borel){\bf 1}_{R}=0,\   \   \   u\cdot {\bf 1}_{R}=\bar{\psi}(u){\bf 1}_{R}\   \   \   \mbox{ for }u\in \Sc.$$
It follows that there is a $\tau$-module isomorphism $\eta$ from $L^0(\psi,0)$ to $V$ with $\eta({\bf 1})={\bf 1}_{R}$.
For $a\in \g,\ \m\in \Z^r$, we have
$$\eta(a_{-1,\m}{\bf 1})=a_{-1,\m}{\bf 1}_{R}=\sum_{i=1}^s \bda_i^\m \left(\vac \otimes \cdots \otimes
		\mathop{a_{-1}\vac}_{\substack{\uparrow \\ i\text{-th}}} \otimes \cdots \otimes \vac\right).$$
	By (\ref{2.4}) we have $$Y(a_{-1,\m}{\bf 1};x_0,\x)=\sum_{m_0\in \Z}a_{m_0,\m}x_0^{-m_0-1}\x^{-\m},$$
	so that
		$$Y(a_{-1,\m}{\bf 1},x_0)=Y(a_{-1,\m}{\bf 1};x_0,\x)|_{\x=1}
		=\sum_{m_0\in \Z}a_{m_0,\m}x_0^{-m_0-1}=a(x_0,\m).$$
On the other hand,  for $a\in \g,\ \m\in \Z^r$, we have
	$$Y(\eta(a_{-1,\m}{\bf 1}),x_0)=\sum_{i=1}^{s}\bda_{i}^{\m}Y(\pi_{i}(a),x_0)=\rho_{R}(a(x_0,\m)),$$
	where $\rho_{R}$ is the Lie algebra homomorphism affording the $\tau$-module $V$ and
	$\pi_{i}$ is the Lie algebra embedding of $\g$ into $\oplus_{j=1}^{s}\g$.
 With $\eta$ a $\tau$-module homomorphism, we have
$$\eta(Y(u,x_0)w)=Y(\eta(u),x_0)\eta(w)\  \  \  \mbox{ for all }u\in L^0(\psi,0)_{(1)},\  w\in L^0(\psi,0).$$
As $L^0(\psi,0)_{(1)}$ generates $L^0(\psi,0)$ as a vertex algebra, it follows that $\eta$ is a homomorphism of vertex algebras.
Thus $\eta$ is an isomorphism of vertex algebras.
Restricted onto $L^0(\psi,0)_{(1)}$, $\eta$ becomes an isomorphism of Lie algebras
onto $\oplus_{i=1}^{s}\g$.
\end{proof}

Note  (see \cite{FZ}, \cite{DL},  \cite{Li1}, \cite{MP1}, \cite{MP2}) that for every positive integer $\ell$,
the module category of $L_{\hat{\g}}(\ell,0)$ is semi-simple
and irreducible $L_{\hat{\g}}(\ell,0)$-modules
are exactly integrable highest weight $\hat{\g}$-modules of level $\ell$.
From \cite{FHL}, \cite{DLM} (Proposition 3.3) and Proposition \ref{SzeroStructProp}, we immediately have:

\begin{coro}\label{SzeroModCateCoro1}
Let $\psi:\Sc\rightarrow \Lr$ be the $\Zr$-graded algebra homomorphism given as in  Proposition \ref{SzeroStructProp}.
Assume that $\ell_1,\dots,\ell_s$ are positive integers.
Then every $\VSzero{\psi}$-module is completely reducible and $\VSzero{\psi}$ has only
 finitely many irreducible modules up to isomorphism.
 \end{coro}

\section{Integrability}

In this section, we give a necessary and sufficient condition that $\Stva{\psi}$ is an integrable $\tau$-module, where
 $\psi:\Sc\rightarrow \Lr $ is a $\Zr$-graded algebra homomorphism
such that $\Image\psi$ is a $\Zr$-graded simple $\Sc$-module.

 Fix a Cartan subalgebra $\h$ of $\g$ and
denote by $\Delta$ and $\Delta_{+}$ the sets of roots and positive roots of $\g$, respectively. Set
$$\g_{\pm}=\sum_{\pm \alpha\in \Delta_{+}}\g_{\alpha}.$$
Let $\theta$ be the highest root of $\g$ and we fix nonzero root vectors
 $e_\theta\in \g_\theta,\ f_{\theta}\in \g_{-\theta}$ such that $\<e_{\theta},f_{\theta}\>=1$.
Let $A$ be any commutative associative algebra with identity (over $\C$).
A $\hat{\g}\otimes A$-module $W$ is said to be {\em integrable}
if $\h$  acts semisimply and $\g_{\alpha}\otimes t_0^{k}u$ acts locally nilpotently on $W$
for any $\alpha\in \Delta,\ k\in \Z,\ u\in A$.

Recall that $L(\psi,0)$ and $L^0(\psi,0)$ are naturally restricted $\tau$-modules, where
$\VSzero{\psi}$ is an irreducible quotient $\tau$-module of $\Stva{\psi}$. We have:

\begin{lem}\label{IntEqualLem}
$L(\psi,0)$ is an integrable $\Toro$-module if and only if $\VSzero{\psi}$ is integrable.
\end{lem}

\begin{proof} Since $\VSzero{\psi}$ is a quotient $\tau$-module of
$\Stva{\psi}$, the ``only if'' part is clear. Let $\pi$ be the natural quotient map from
$\Stva{\psi}$ to $\VSzero{\psi}$.  By Lemma \ref{LrLem}, we have an injective \rtva{} homomorphism $\tilde{\pi}$  from
$\Stva{\psi}$ to $\lr{\VSzero{\psi}}$. As $\tilde{\pi}$ is injective,
$\Stva{\psi}$ is a $\tau$-submodule of $\lr{\VSzero{\psi}}$.
It can be readily seen that
$\lr{\VSzero{\psi}}$ is an integrable $\tau$-module if and only if
 $\VSzero{\psi}$ is an integrable $\tau$-module. Then the ``if'' part follows.
\end{proof}

Let $\{e,f,h\}$ be the standard basis of $\mathfrak{sl}_2$. Consider the $r$-loop algebra
$L_r(\mathfrak{sl}_2)$. We equip $L_r(\mathfrak{sl}_2)$ with the $\Z$-grading given by
$$\deg L_r(\C e)=1,  \  \   \deg (L_r(\C f))=-1, \   \   \deg(L_r(\C h))=0,$$
to make $L_r(\mathfrak{sl}_2)$ a $\Z$-graded Lie algebra.
Notice that Lie algebra $\lr{\C e\oplus\C h}$
is the semi-product of abelian subalgebra $L_{r}(\C h)$  with ideal $L_{r}(\C e)$.

Let $\phi:\lr{\C h}\rightarrow \C$ be a linear function.
We have a $1$-dimensional $\lr{\C e\oplus\C h}$-module, denoted by $\C_{\phi}$, where $\C_{\phi}= \C$ and
\begin{eqnarray}
(h\otimes \rsymbol{t}^\rsymbol{m})\cdot 1=\phi(h\otimes \rsymbol{t}^\rsymbol{m})1,\    \    \
(e\otimes \rsymbol{t}^\rsymbol{m})\cdot 1=0\   \  \   \mbox{for }\ranger{m}.
\end{eqnarray}
Form an induced module
\begin{eqnarray}\label{eq:defofUpsi}
U(\phi)= U(\lr{\mathfrak{sl}_2})\otimes_{U(\lr{\C e\oplus\C h})}\C_{\phi}.
 \end{eqnarray}
 Let $1_{\phi}$ denote the generator $1\otimes 1$. From the P-B-W theorem we have
 $U(\phi)=U(L_r(\C f))1_{\phi}$.
It then follows that $h$ acts semisimply on $U(\phi)$ with
$$U(\phi)=\bigoplus_{n\in \N}U(\phi)_{\phi(h)-2n},$$
where $U(\phi)_{\lambda}=\{ u\in U(\phi) \  |\  hu=\lambda u\}$ for $\lambda\in \C$.
Consequently, $U_{\phi}$
has a unique maximal $\lr{\mathfrak{sl}_2}$-submodule.
Denote by $L_\phi$ the unique simple quotient $\lr{\mathfrak{sl}_2}$-module of $U(\phi)$.

\begin{lem}\label{preparation}
Let $I$ be the sum of ideals $J$  of $\Lr$ such that $\phi(h \otimes J)=0$.
Then, for $a\in\Lr$,  $\phi(h\otimes a\Lr)=0$ if and only if $a\in I$,
 and if and only if $(f\otimes a)1_{\phi}=0$. On the other hand, $\dim (L_r/I)<\infty$ if and only if
 $\dim (L_r/\sqrt{I})<\infty$, where
\begin{eqnarray}
\sqrt{I}=\left\{ a\in \Lr\,|\,a^n\in I \text{ for some positive integer }n \right\}.
\end{eqnarray}
\end{lem}

\begin{proof} The first part is clear as $aL_r$ is an ideal containing $a$. If $(f\otimes a)1_{\phi}=0$, for any $b\in L_r$ we have
$$\phi(h\otimes ab)1_{\phi}=[e\otimes b,f\otimes a]1_{\phi}=(e\otimes b)(f\otimes a)1_{\phi}-(f\otimes a)(e\otimes b)1_{\phi}=0,$$
which implies $aL_r\subset I$, and hence $a\in I$. Conversely, assume $a\in I$. We have
$$(e\otimes b)(f\otimes a)1_{\phi}=(f\otimes a)(e\otimes b)1_{\phi}+\phi(h\otimes ab)1_{\phi}=0$$
for all $b\in L_r$. That is, $L_r(\C e)(f\otimes a)1_{\phi}=0$. As $L_{\phi}$ is irreducible, we must have $(f\otimes a)1_{\phi}=0$.

As for the last assertion, assume $\dim (L_r/\sqrt{I})<\infty$.
As $\dim (\C[t_i]+\sqrt{I})/\sqrt{I}<\infty$ for each $1\le i\le r$, there are nonzero polynomials $p_1(t),\dots,p_r(t)\in \C[ t]$ such that
\[
\left( p_1(t_1),\dots,p_r(t_r) \right)\subset \sqrt{I}.
\]
Then there exists a positive integer $n$ such that
\[
\left(p_1(t_1)^n, \dots, p_r(t_r)^n \right)\subset I.
\]
Thus, $\dim (L_r/I)<\infty$. The other direction is clear as $I\subset \sqrt{I}$.
\end{proof}

We have:

\begin{prop}\label{HwIntModLem}
$L_\phi$ is an integrable $\lr{\mathfrak{sl}_2}$-module if and only if $\dim L_{\phi}<\infty$.
\end{prop}

\begin{proof} For any $\ranger{m}$, set
$$\mathfrak{sl}_2^{(\m)}=\C (e\otimes \rsymbol{t}^\rsymbol{m})+ \C (f\otimes \degr{t}{m})+ \C h,$$
which is canonically isomorphic to $\mathfrak{sl}_2$.
It can be readily seen that $L_\phi$ is integrable if $\dim L_{\phi}<\infty$.

Now, we assume that $L_\phi$ is integrable.
Recall from Lemma \ref{preparation} the ideal $I$ of $L_r$. We next prove that $I$ is co-finite dimensional.
By Lemma \ref{preparation}, we need to prove that
$\sqrt{I}$ is co-finite dimensional. Note that
\begin{eqnarray*}
\sqrt{I}=\mathop{\bigcap}\limits_{i=1}^d Q_i,
\end{eqnarray*}
where $Q_i$  for $1\leq i\leq d$ are prime ideals of $\Lr$, so that
 $\Lr/Q_i$ are integral ring extensions over $\C$.

We first consider the case that all $L_r/Q_i$ are finite dimensional.
In this case, every $L_r/Q_i$ is a finite-dimensional field extension of $\C$, which must be $1$-dimensional. Thus
$Q_i=\left( t_1-c_{i1},\dots,t_r-c_{ir} \right)$ for some non-zero complex numbers $c_{i1},\dots, c_{ir}$, for $1\leq i\leq d$.
Consequently,
\[
\left( \prod\limits_{i=1}^d (t_1-c_{i1}),\dots,\prod\limits_{i=1}^d (t_r-c_{ir}) \right)
\subset \mathop{\bigcap}\limits_{i=1}^d Q_i=\sqrt{I}.
\]
Thus $\sqrt{I}$ is co-finite dimensional, and so is $I$.

Now, we prove that all $L_r/Q_i$ must be finite dimensional. Otherwise, $\Lr\big/ Q_i$ is infinite dimensional for some $i$.
As $\Lr\big/ Q_i$ is also an infinite dimensional integral domain (over $\C$), there exists $y\in L_r$ such that $y+Q_i$ is a transcendental element of $L_r/Q_i$
over $\C$.
Then $y^{n}$ for $n\in \N$ are linearly independent in $\Lr$ modulo $Q_i$, which
implies that $y^{n}$ for $n\in \N$ are linearly independent in $\Lr$ modulo $I$.
By Lemma \ref{preparation}, $(f\otimes y^{n})1$ with $n\in \N$ are linearly
independent in $L_\phi$.
For each positive integer $n$, we set
\begin{eqnarray*}
L(n)=\mathop{\bigoplus}\limits_{0\leq m<n} \C \left( f\otimes y^m \right)1_\phi\subset L_{\phi}.
\end{eqnarray*}

We claim that for each positive integer $n$ and for any positive integers $i_1,\dots,i_n$,
\begin{eqnarray}\label{EFActionEq}
&&\left( e\otimes 1 \right)^{n-1}\prod\limits_{j=1}^n (f\otimes y^{i_j})1_\phi
 \equiv c_n \left(f\otimes y^{N}\right)1_\phi \ \ \mbox{ modulo } L(N),
\end{eqnarray}
where $N=i_1+\cdots +i_n$ and $c_n=(-1)^{n-1}n!(n-1)!$.
It is clear that (\ref{EFActionEq}) is true when $n=1$.
Assume that $n$ is a positive integer such that (\ref{EFActionEq}) holds for
each positive integer less than or equal to $n$.
For positive integers $i_1,\dots,i_{n+1}$ and any $1\leq k<\ell\leq n+1$, we set $N'=\sum\limits_{j=1}^{n+1}i_j$ and set
\begin{eqnarray*}
A(k)=\prod\limits_{\substack{1\leq j\leq n+1\\j\neq k}}(f\otimes y^{i_j}) &
\mbox{and} & A(k,\ell)=\prod\limits_{\substack{1\leq j\leq n+1\\j\neq k,\ell}}(f\otimes y^{i_j}).
\end{eqnarray*}
Then
\begin{eqnarray*}
&&\left( e\otimes 1 \right)^{n}\prod\limits_{j=1}^{n+1} (f\otimes y^{i_j})1_\phi
	\equiv\left( e\otimes 1 \right)^{n-1}\left[e\otimes 1,\prod\limits_{j=1}^n (f\otimes y^{i_j})\right]1_\phi\\
	&&\quad\equiv\left( e\otimes 1 \right)^{n-1}\sum\limits_{k=1}^{n+1}\varphi(h\otimes y^{i_k})A(k)1_\phi
		-2\left( e\otimes 1 \right)^{n-1}\sum\limits_{1\leq k<\ell\leq n+1}A(k,\ell)(f\otimes y^{i_k+i_\ell})1_\phi\\
	&&\quad \equiv c_n \sum\limits_{k=1}^{n+1}\varphi(h\otimes y^{i_k}) \left(f\otimes y^{M_k}\right)1_\phi
		-2c_n\sum\limits_{1\leq k<\ell\leq n+1}\left(f\otimes y^{N'}\right)1_\phi\ \ \ \ \mbox{modulo }L(N''),
\end{eqnarray*}
where $M_k=\sum\limits_{j\neq k}i_j<N'$ and $L(N'')=\sum\limits_{k=1}^{n+1}L(M_k)+L(N')=L(N')$.
Therefore,
\begin{eqnarray*}
&&\left( e\otimes 1 \right)^{n}\prod\limits_{j=1}^{n+1} (f\otimes y^{i_j})1_\phi\equiv -c_n (n+1) n \left(f\otimes y^{N'}\right)1_\phi\\
&&\quad\equiv c_{n+1} \left(f\otimes y^{N'}\right)1_\phi \ \ \ \ \mbox{modulo }L(N').
\end{eqnarray*}
This proves that (\ref{EFActionEq}) holds for all positive integers $n$.

As $L_\phi$ is integrable, there exists a positive integer $n$ such that $(f\otimes y)^n 1_\phi=0$.
Then
\begin{eqnarray*}
0\equiv(e\otimes 1)^{n-1}(f\otimes y)^n 1_\phi\equiv c_n (f\otimes y^n)1_\phi\not\equiv 0\ \ \ \ \mbox{modulo }L(n).
\end{eqnarray*}
This contradiction implies that all $\Lr/Q_i$ must be finite dimensional.
Therefore, $I$ is co-finite dimensional.
With $U(L_r(\C f))$ commutative, we have
 $$L_{\phi}=U(L_r(\mathfrak{sl}_2))1_{\phi}=U(L_r(\C f))1_{\phi}=U(\C f\otimes (L_r/I))1_{\phi}.$$
Since $\dim (L_r/I)<\infty$ and for every $\m\in \Z^r$, $f\otimes {\bf t}^{\m}$ is nilpotent on $1_{\phi}$,
it follows that  $L_\phi$ is finite dimensional.
\end{proof}

Now, we have:

\begin{thm}\label{IntegrableProp123}
Let $\psi:\Sc\rightarrow \Lr$ be a $\Zr$-graded algebra homomorphism
such that $\Image\psi$ is a $\Zr$-graded simple $\Sc$-module. Then
$L(\psi,0)$  is an integrable $\Toro$-module if and only if there exist finitely many positive integers $\ell_{1},\dots,\ell_{s}$
and distinct vectors $\bda_1,\dots,{\bf a}_{s}\in (\C^{\times})^r$
such that
\begin{eqnarray*}
\psi(\cent\otimes\bdt^\m)=\left(\sum\limits_{i=1}^{s} \ell_i \bda^{\m}\right)\bdt^\m \ \ \ \ \mbox{for } \m\in\Zr.
\end{eqnarray*}
\end{thm}

\begin{proof}
The ``if'' part  follows from Proposition \ref{SzeroStructProp} and Lemma \ref{IntEqualLem}.
Now, we assume that $L(\psi,0)$ is integrable. Then  $\VSzero{\psi}$ is integrable by Lemma \ref{IntEqualLem}.
We fix an $\alpha\in \Delta$ and choose $e\in\g_\alpha$, $f\in\g_{-\alpha}$ such that
$[e,f]=\check{\alpha}$ and $\<e,f\>=1$.
Set
\[
\s=\C (e\otimes t_0)+ \C (f\otimes t_0^{-1})+ \C(\check{\alpha}+\cent).
\]
Notice that $\s\cong \mathfrak{sl}_2$.
Then $U(\Lr(\s))\vac$ is an integrable $\Lr(\s)$-submodule of $\VSzero{\psi}$.
Define a $\Zr$-graded algebra homomorphism $\psi_{\alpha}:U(\lr{\check{\alpha}+\cent})\rightarrow \Lr$ by
$$\psi_{\alpha}((\check{\alpha}+\cent)\otimes \bdt^\m)= \psi(\cent\otimes \bdt^\m)\   \   \   \mbox{ for }\m\in \Z^r.$$
Then $\Image\psi_{\alpha}$ is a $\Zr$-graded simple $U(\lr{\check{\alpha}+\cent})$-module.
Set $\phi=E(1)\circ \psi_{\alpha}$.
We have $L_r(e\otimes t_0)\vac=0$ and $X\cdot \vac=\phi(X)\vac$ for $X\in U(\lr{\check{\alpha}+\cent})$,
noticing that $(\check{\alpha}\otimes \bdt^{\m})\cdot \vac=0$ for $\m\in \Z^r$.
It follows that  $L_{\phi}$ is a homomorphism image of $U(\Lr(\s))\vac$.
Consequently, $L_{\phi}$ is integrable. By Proposition \ref{HwIntModLem}, $L_{\phi}$ is finite dimensional.
From the proof of Proposition 3.20 in \cite{R2},  there exist finitely many nonzero dominant integral weights $\lambda_{1},\dots,\lambda_{s}$ of $\s$ and
distinct vectors $\bda_1,\dots,{\bf a}_{s}\in (\C^{\times})^r$ such that
\begin{eqnarray*}
\phi((\check{\alpha}+\cent)\otimes \bdt^\m)=\sum\limits_{i=1}^s \bda_i^\m \lambda_{i}(\check{\alpha}+\cent)
\end{eqnarray*}
for $\m\in \Z^r$. Thus
\begin{eqnarray*}
\psi(\cent\otimes\bdt^\m)=\left(\sum\limits_{i=1}^{s} \ell_i \bda_i^{\m}\right)\bdt^\m \ \ \ \ \mbox{for } \m\in\Zr,
\end{eqnarray*}
where $\ell_i=\lambda_i(\check{\alpha}+\cent)$.
\end{proof}

\section{Irreducible $L(\psi,0)$-modules}

In this section, we determine irreducible $L(\psi,0)$-modules with $L(\psi,0)$ an integrable $\tau$-module.




For this section, we assume that  $\psi:\Sc\rightarrow \Lr$ is the $\Zr$-graded algebra homomorphism
defined in (\ref{IntCentActEq}),
associated to positive integers $\ell_1,\dots,\ell_s$ and distinct vectors
\begin{eqnarray}
\bda_i=\left( a_{i1},\dots,a_{ir} \right)\in (\C^{\times})^{r}
\end{eqnarray}
for $1\leq i\leq s$. We next present some basic results about  $\Image \psi$. First, we have:

\begin{lem}\label{ImPsiLem}
There exist positive integers $\nu_1,\dots,\nu_r$ such that
\begin{eqnarray}
\C\left[ t_1^{\pm \nu_1},\dots, t_r^{\pm \nu_r} \right]\subset \Image \psi.
\end{eqnarray}
\end{lem}

\begin{proof} Note that $\Image \psi$ is a subalgebra of $L_r$. Then it suffices to prove that for each $1\leq j\leq r$,
there exists a positive integer $\nu_j$ such that $t_{j}^{\pm \nu_j}\in \Image \psi$. Furthermore,
 it suffices to prove that for each $1\leq j\leq r$,
there exist a positive integer $p_j$ and a negative integer $n_{j}$ such that $t_{j}^{p_j}, \ t_j^{n_{j}}\in \Image \psi$.
Fix $1\leq j\leq r$. Recall that $\psi(\cent\otimes t_j^m)=\left(\sum\limits_{i=1}^{s}\ell_i a_{ij}^m\right) t_j^m$ for $m\in \Z$.
Then it suffices to prove that  $\sum\limits_{i=1}^{s}\ell_i a_{ij}^m\ne 0$ for some positive integer $m$ and for some negative integer $m$.
Let $S_{1}\cup \cdots \cup S_{p}$ be the partition of the set $\{1,2,\dots,s\}$ corresponding to the equivalence relation given by
$\mu \equiv \nu$ if and only if $a_{\mu,j}=a_{\nu,j}$. For $1\le q\le p$, set $b_{q}=a_{ij}$ for some $i\in S_{q}$. Then
$$\sum\limits_{i=1}^{s}\ell_i a_{ij}^m=\left(\sum_{i\in S_{1}}\ell_{i}\right)b_{1}^{m}+\dots +\left(\sum_{i\in S_{p}}\ell_{i}\right)b_{p}^{m}.$$
Since $b_{1},\dots,b_{p}$ are distinct nonzero complex numbers, the system of linear equations
$$b_{1}^{m}x_{1}+\dots +b_{p}^{m}x_{p}=0$$
for $m=1,2,\dots$ does not have nontrivial solutions.
As $\ell_1,\dots,\ell_s$ are positive integers,  we must have that $\sum\limits_{i=1}^{s}\ell_i a_{ij}^m\ne 0$ for some positive integer $m$.
It is also clear that $\sum\limits_{i=1}^{s}\ell_i a_{ij}^m\ne 0$ for some negative integer $m$.
\end{proof}

As $\Image\psi$ is a $\Z^r$-graded simple subalgebra of $L_r$, we immediately have:

\begin{coro}\label{coro:ImPsiGp}
Set
\begin{eqnarray}
\Lambda(\psi)=\{\m\in\Zr\,\mid\, \bdt^\m\in\Image\psi\}.
\end{eqnarray}
Then $\Lambda(\psi)$ is a cofinite subgroup of $\Zr$.
\end{coro}

\begin{de}
For each $1\le j\le r$, let $\nu_j$ be the least positive integer such that $t_j^{\pm \nu_{j}}\in \Image \psi$
(see Lemma \ref{ImPsiLem}).
Furthermore, for each $1\le j\le r$, set
$$S_{j}=\{ a_{ij}^{\nu_j}\ |\ i=1,2,\dots,s\}$$
(multiplicity-free) and define
\begin{eqnarray}\label{pdefi}
p_{j}(t_j)=\prod\limits_{c\in S_j}\left(t_j^{\nu_j}-c\right).
\end{eqnarray}
\end{de}

From definition we have
\begin{eqnarray}
p_{j}(t_j)\in\C\left[ t_1^{\pm \nu_1},\dots, t_r^{\pm \nu_r} \right]\subset \Image \psi.
\end{eqnarray}

\begin{rem}\label{explaination}
{\em Note that any $V(\Sc,0)$-module $W$ is naturally an $\Sc$-module and $\ker \psi\subset \Sc$.
We have that a $V(\psi,0)$-module amounts to a $V(\Sc,0)$-module $W$ such that
$(\ker \psi)W=0$. On the other hand, with $\Sc$ a subalgebra of $U(\tau)$, every $\tau$-module $W$ is also naturally an $\Sc$-module, and furthermore,
if $(\ker \psi)W=0$, then $W$ is naturally an $({\rm Im} \psi)$-module as $\Sc/\ker \psi\cong {\rm Im} \psi$.}
\end{rem}

\begin{rem}\label{integrable-module}
{\em We here present a simple fact about integrable $\hat{\g}$-modules.
Let $W$ be a nonzero restricted $\hat{\g}$-module of level $\ell$.
Let $\theta$ be the highest root of $\g$ and let $e_\theta$ be a non-zero vector in $\g_\theta$.
Assume that
$e_{\theta}(x)^{k}=0$ for some positive integer $k$. Then  we show that $W$ is an integrable $\hat{\g}$-module as follows:
(1) It is known that $W$ is naturally a module for vertex algebra $V_{\hat{\g}}(\ell,0)$ (cf. \cite{Li1}).
(2) Let $J$ be the annihilating ideal of $W$ and set
$\overline{V}=V_{\hat{\g}}(\ell,0)/J$.
Then $W$ is a faithful $\overline{V}$-module.
As $Y_{W}(e_{\theta}(-1)^{k}{\bf 1},x)=Y_{W}(e_{\theta},x)^{k}=e_{\theta}(x)^{k}=0$, we have
$e_{\theta}(-1)^{k}{\bf 1}=0$ in $\overline{V}$.
Then $\overline{V}$ is an integrable $\hat{\g}$-module (by using the Chevalley generators).
(3) As $W\ne 0$, we must have $\overline{V}\ne 0$. It follows that
$\overline{V}=L_{\hat{\g}}(\ell,0)$, which implies that
$L_{\hat{\g}}(\ell,0)$ is integrable and  $\ell$ must be a nonnegative integer (see \cite{K}).
 (4)  As $W$ is a weak $L_{\hat{\g}}(\ell,0)$-module, from \cite{DLM}
$W$ must be an integrable $\hat{\g}$-module.}
\end{rem}

The following is a characterization of $L^0(\psi,0)$ as a quotient module of $V^{0}(\psi,0)$:

\begin{prop}  Set $\ell=\ell_1+\dots+\ell_s$. Then $L^0(\psi,0)=V^0(\psi,0)/I$, where
$I$ is the ideal of vertex algebra $\Vzero{\psi}$ generated by
\begin{eqnarray}
 e_\theta(-1,\m)^{\ell+1}\vac,\  \  \  \   (u\otimes t_0^{-1}f(\bdt))\vac
\end{eqnarray}
for $\ranger{m},\  u\in \g,~ f(\bdt)\in p_1(t_1)L_r+\cdots+p_r(t_r)L_r$.
\end{prop}

\begin{proof} Let $\pi$ denote the quotient map from $V^0(\psi,0)$ to $L^0(\psi,0)$,
which is also a $\tau$-module homomorphism.
 Recall from Proposition \ref{SzeroStructProp} the explicit construction of $L^0(\psi,0)$ in terms of highest weight irreducible $\hat{\g}$-modules $L_{\hat{\g}}(\ell_i,0)$ for $i=1,\dots,s$. Let $\Delta_i$ denote the natural algebra map from $U(\hat{\g})$ to
 $U(\hat{\g})^{\otimes s}$. We have
 $$\pi\left(e_\theta(-1,\m)^{\ell+1}\vac\right)=e_\theta(-1,\m)^{\ell+1}\pi(\vac)=\left(\sum_{i=1}^{s}\bda_i^{\m}\Delta_i(e_{\theta}(-1))\right)^{\ell+1}\pi({\bf 1})=0$$
 as $\Delta_i(e_{\theta}(-1))^{\ell_i+1}\pi({\bf 1})=0$ for $1\le i\le s$.
 On the other hand, for $u\in \g$, as
 $$u(x_0,\x)\pi(\vac)=\sum_{i=1}^{s}\delta\left(\frac{\bda_i}{\x}\right)\Delta_i(u(x_0))\pi(\vac),$$
 where $\delta\left(\frac{\bda_i}{\x}\right)=\prod_{j=1}^{r}\delta\left(\frac{a_{ij}}{x_j}\right)$,
 we have
 $$p_j(x_j)u(x_0,\x)\pi(\vac)
 =\sum_{i=1}^{s}p_j(x_j)\delta\left(\frac{\bda_i}{\x}\right)\Delta_i(u(x_0))\pi(\vac)=0$$
 for $1\le j\le r$.  Thus  for any
 $f(\bdt)\in p_1(t_1)L_r+\cdots+p_r(t_r)L_r$, we have
 $f(\x)u(x_0,\x)\pi(\vac)=0,$
 which implies $(u\otimes t_0^{n_0}f(\bdt))\pi(\vac)=0$ for all $n_0\in \Z$.
It now follows that $I\subset \ker \pi$.

From definition, the following relations hold  in $\Vzero{\psi}\big/I$:
\begin{eqnarray*}
&&Y^0\left(e_\theta(-1,\m)\vac,x_0\right)^{\ell+1}=Y^{0}\left(e_\theta(-1,\m)^{\ell+1}\vac,x_0\right)=0, \\
&&\sum\limits_{k\in \Z}u\otimes t_0^{k}f(\bdt) x_0^{-k-1}=Y^{0}((u\otimes t_0^{-1}f(\bdt))\vac,x_0)=0
\end{eqnarray*}
for  $\ranger{m}$, $u\in \g$, $f(\bdt)\in p_1(t_1)L_r+\cdots+p_r(t_r)L_r$. Set
$$K=L_r\big/(p_1(t_1)L_r+\cdots+p_r(t_r)L_r).$$
The second relation implies that $\Vzero{\psi}\big/I$ is a restricted  module for the quotient algebra
$\hat{\g}\otimes K$ of $\tau$. As $K\cong \C^{\oplus N}$ with $N=\prod_{j=1}^s\deg p_j(t_j)$,
\begin{eqnarray*}
\hat{\g}\otimes K\cong \mathop{\oplus}_{N-\text{copies}} \hat{\g}
\end{eqnarray*}
canonically. It then follows from the first relation that $\Vzero{\psi}\big/I$ is an integrable
module for the quotient algebra  (see Remark \ref{integrable-module}) and then for $\tau$.
Then $\Vzero{\psi}\big/I$ is a completely reducible module for the quotient algebra (and then for $\Toro$).
Since $\Vzero{\psi}\big/I$, a nonzero $\Toro$-module, is generated by $\vac$ (a highest weight vector),
it follows that $\Vzero{\psi}\big/I$ is irreducible.
Consequently, $\Vzero{\psi}\big/I=\VSzero{\psi}$.
\end{proof}

Next, we shall determine $L(\psi,0)$ and its irreducible modules in terms of $L^0(\psi,0)$ and its irreducible modules.
Recall that $L^0(\psi,0)$ is the $\Z$-graded simple quotient vertex algebra of $L(\psi,0)$,
 $\Vzero{\psi}$ is the \vqa{} of $V(\psi,0)$ modulo the ideal generated by
$\bdt^\m-\vac$ for $\m\in \Z^r$ with $\bdt^\m\in \Image\psi$. We have the following commutative diagram
\begin{eqnarray}\label{ecomm-diagram1}
\begin{matrix} V(\psi,0)& \overset{\gamma}{\longrightarrow} & V^0(\psi,0)\\
\downarrow \alpha &\mbox{} & \downarrow \alpha^{0}\\
L(\psi,0)&\overset{\pi}{\longrightarrow} &L^0(\psi,0).
\end{matrix}
\end{eqnarray}
Note that $\alpha,\alpha^{0},\gamma$ and $\pi$ are all canonical quotient maps, which are
homomorphisms of vertex algebras  and $\tau$-modules. Furthermore,
using Lemma \ref{LrLem} we obtain the following  commutative  diagram
\begin{eqnarray}\label{ecomm-diagram2}
\begin{matrix} V(\psi,0)& \overset{\tilde{\gamma}}{\longrightarrow} & L_r(V^0(\psi,0))\\
\downarrow \alpha &\mbox{} & \downarrow \alpha^{0}\otimes 1\\
L(\psi,0)&\overset{\tilde{\pi}}{\longrightarrow} &L_r(L^0(\psi,0)).
\end{matrix}
\end{eqnarray}
From Lemma \ref{LrLem}, $\tilde{\pi}$ is injective.

Note that
\begin{eqnarray}
V(\psi,0)&\cong &U(L_r(\g\otimes t_0^{-1}\C[t_0^{-1}]))\otimes \Image\psi,\\
\Vzero{\psi}&\cong & U(L_r(\g\otimes t_0^{-1}\C[t_0^{-1}]))
\end{eqnarray}
 as vector spaces.
 In view of this, we shall always view $U(L_r(\g\otimes t_0^{-1}\C[t_0^{-1}]))$ as a subspace of $V(\psi,0)$ and $\Vzero{\psi}$.
Since $L_r(\g\otimes t_0^{-1}\C[t_0^{-1}])$ is a $\Zr$-graded Lie algebra,
$U(L_r(\g\otimes t_0^{-1}\C[t_0^{-1}]))$ is an associative $\Zr$-graded algebra.
For $\n\in\Zr$, denote by $U(\n)$ the degree $\n$ subspace of $U(L_r(\g\otimes t_0^{-1}\C[t_0^{-1}]))$.

We have
 \begin{eqnarray}
 \gamma(X\otimes \bdt^{\m})=X\   \   \
 \mbox{ for }X\in U(L_r(\g\otimes t_0^{-1}\C[t_0^{-1}])),\ \m\in\Lambda(\psi).
 \end{eqnarray}
For  the homomorphism $\tilde{\gamma}: V(\psi,0)\rightarrow L_r(V^0(\psi,0))$,
we have
 \begin{eqnarray}
 \tilde{\gamma}(X\otimes \bdt^{\m})=X\otimes \bdt^{\m+\n}
 \end{eqnarray}
for $X\in U(\n),\ \m\in \Lambda(\psi)$ with $\n\in \Z^r$.
Furthermore, we have:

\begin{lem}\label{lem:tildegammaProp}
The map $\tilde{\gamma}$ is injective.
On the other hand, we have
\begin{eqnarray}
\Image\tilde{\gamma}={\rm span}\left\{ X\otimes \bdt^{\m}\  |\   X\in U(\n),\  \m,\n\in \Z^r\  \mbox{with }\m-\n\in \Lambda(\psi)\right\}.
\end{eqnarray}
\end{lem}

\begin{proof}  Since $\tilde{\gamma}$ preserves the $\Z^r$-grading, it suffices to show that $\tilde{\gamma}(X)\ne 0$ for any
nonzero homogeneous vector $X\in V(\psi,0)$. Assume $\deg X=\m$. Then
$$X=X_1 \otimes \bdt^{\m^1}+\cdots +X_k\otimes \bdt^{\m^k},$$
where $\m^{1},\dots,\m^{k}\in
\Lambda(\psi)$ are distinct and $X_{1},\dots,X_k$ are nonzero homogeneous
elements of $U(L_r(\g\otimes t_0^{-1}\C[t_0^{-1}]))$ of degrees $\n^1,\dots,\n^k$ such that $\m^i+\n^i=\m$.
As $\m^{1},\dots,\m^{k}$ are distinct,
$\n^1,\dots,\n^k$ must be distinct, so that $X_1,\dots,X_k$ are linearly independent. Thus
$$\tilde{\gamma}(X)=\tilde{\gamma}(X_1\otimes \bdt^{\m^1}+\cdots +X_k\otimes \bdt^{\m^k})
=(X_1+\dots+X_k)\otimes \bdt^{\m}\ne 0.$$
This shows that $\tilde{\gamma}$ is injective.

Note that $X\otimes \bdt^{\m}\in \Image\tilde{\gamma}$
for any $X\in U(\n),\ \m,\n\in \Z^r$ with $\m-\n\in \Lambda(\psi)$ as
\begin{eqnarray*}
\tilde{\gamma}\left(X\otimes \bdt^{\m-\n} \right)=\gamma(X\otimes \bdt^{\m-\n})\otimes \bdt^{\m}
=X\otimes \bdt^\m.
\end{eqnarray*}
On the other hand, let $X$ be a nonzero homogeneous element of $\Image\tilde{\gamma}$  of
degree $\m\in \Z^r$. As $X\in L_r(V^0(\psi,0))$, we have
$X=\sum_{i=1}^k X_i\otimes \bdt^\m$,
where $X_i\in U(L_r(\g\otimes t_0^{-1}\C[t_0^{-1}]))\subset \Vzero{\psi}$ are nonzero homogeneous of
distinct degrees $\n_i$.  Since $X\in \Image\tilde{\gamma}$, there exist nonzero homogeneous
$Y_1,\dots,Y_\ell\in U(L_r(\g\otimes t_0^{-1}\C[t_0^{-1}]))$ of distinct degrees
$\rsymbol{p}_j$ with $\m-\rsymbol{p}_j\in \Lambda(\psi)$ such that
\begin{eqnarray*}
\sum\limits_{i=1}^k X_i\otimes
\bdt^\m=\tilde{\gamma}\left(\sum\limits_{j=1}^\ell Y_j\otimes
\bdt^{\m-\rsymbol{p}_j}\right)=\sum\limits_{j=1}^\ell Y_j\otimes \bdt^\m.
\end{eqnarray*}
Since $X_1,\dots,X_k$ and $Y_1,\dots,Y_\ell$ are all homogeneous and nonzero,
we must have $\{\n_i\}_{i=1}^k=\{{\bf p}_j\}_{j=1}^\ell$.
Consequently, $\m-\n_i\in \Lambda(\psi)$. This proves that $X$ is contained in the span.
\end{proof}

We shall need the following technical result:

\begin{lem}\label{lem:IdealCorrespLem222222}
Let $J$ be an ideal of $\T{\psi}$. Then
\begin{eqnarray}
\tilde{\gamma}(J)=(\Image\tilde{\gamma})\cap\lr{\gamma(J)}.
\end{eqnarray}
\end{lem}

\begin{proof}
From Corollary \ref{IdealCriLem},  $J$ is automatically $\Zr$-graded.
From the definition of $\tilde{\gamma}$ we have
$\tilde{\gamma}(J)\subset (\Image\tilde{\gamma})\cap \lr{\gamma(J)}$.
To prove the converse inclusion, let $X$ be a
nonzero homogeneous element in $(\Image\tilde{\gamma})\cap \lr{\gamma(J)}$ of
degree $\m$. Then
$X=\sum_{i=1}^k X_i\otimes \bdt^{\m}$,
where $X_i$ are nonzero homogeneous elements in
$U(L_r(\g\otimes t_0^{-1}\C[t_0^{-1}]))$ of distinct degrees $\n_i$ such that
$\sum_{i=1}^k X_i\in \gamma(J)$.
Let $Y\in J$ such that $\gamma(Y)=\sum_{i=1}^k X_i$.
As $J$ is $\Zr$-graded, we have
$Y=Y_1+\cdots+Y_\ell$, where $Y_j\in J$ are nonzero homogeneous elements of
distinct degrees $\rsymbol{p}_j$.
Furthermore, for each $1\le j\le \ell$, we have
\begin{eqnarray*}
Y_j=\sum\limits_{s=1}^{N_j}Y_{js}\otimes \bdt^{\rsymbol{p}_j-\m_{js}},
\end{eqnarray*}
where $Y_{js}\in U(L_r(\g\otimes t_0^{-1}\C[t_0^{-1}]))$ are nonzero homogeneous
elements of distinct degrees $\m_{js}$ such that
$\rsymbol{p}_j-\m_{js}\in \Lambda(\psi)$ for all $1\leq
s\leq N_j$.

Let $U_{\m}$ be the
subspace of $U(L_r(\g\otimes t_0^{-1}\C[t_0^{-1}]))$, spanned by homogeneous elements of degree
$\n\in\Lambda(\psi)+\rsymbol{m}$, and let
$U_{\rsymbol{m}}'$ be the subspace spanned by homogeneous elements of
degree $\n\not\in\Lambda(\psi)+\rsymbol{m}$.
Then
\begin{eqnarray*}
U(L_r(\g\otimes t_0^{-1}\C[t_0^{-1}]))=U_{\m}\oplus U_\m'.
\end{eqnarray*}
Set
\begin{eqnarray*}
Y_\m=\sum\limits_{\substack{1\leq j\leq \ell\\ \rsymbol{p}_j\in \Lambda(\psi)+\m
}}Y_j\in J&\andtext\,\,\,\,
Y_{\m}'=\sum\limits_{\substack{1\leq j\leq \ell\\ \rsymbol{p}_j\not\in
\Lambda(\psi)+\m }}Y_j\in J.
\end{eqnarray*}
Then $\gamma(Y_\m)\in U_\m$ and $\gamma(Y_\m')\in U_\m'$.
Since
\begin{eqnarray*}
\gamma(Y_\m)+\gamma(Y_\m')=X_1+\cdots+X_k\in U_\m,
\end{eqnarray*}
we get that  $\gamma(Y_\m')=0$ and
$\gamma(Y_\m)=X_1+\cdots+X_k$.
Thus, we may assume $Y=Y_\m$ in the first place.
From the definition of $Y_\m$, we have $\rsymbol{p}_j-\m\in \Lambda(\psi)$
for $1\leq j\leq n$.
Then $\bdt^{\m-\rsymbol{p}_j}\in \Image\psi\subset
\T{\psi}$. As $J$ is an ideal, we have $\sum_{j=1}^n \bdt^{\m-\rsymbol{p}_j}\cdot Y_j\in J$
and
\begin{eqnarray*}
\tilde{\gamma}\left(\sum\limits_{j=1}^n \bdt^{\m-\rsymbol{p}_j}\cdot Y_j \right)
=\gamma\left(\sum\limits_{j=1}^n \bdt^{\m-\rsymbol{p}_j}\cdot Y_j\right)\otimes \bdt^\m=\gamma(Y)\otimes
\bdt^\m=\sum\limits_{i=1}^k X_i\otimes \bdt^\m.
\end{eqnarray*}
This proves $X=\sum\limits_{i=1}^k X_i\otimes \bdt^\m\in \tilde{\gamma}(J)$.
Thus $(\Image\tilde{\gamma})\cap \lr{\gamma(J)}\subset \tilde{\gamma}(J)$.
Therefore, $(\Image\tilde{\gamma})\cap \lr{\gamma(J)}= \tilde{\gamma}(J)$.
\end{proof}

Note that for $\m\in \Z^r$, if $\bdt^\m\in \Image\psi$,
then $\bdt^{-\m}\in \Image\psi$.
For $a\in \g$, $\ranger{n}$, and $f(\bdt)=\sum\limits_{\ranger{m}}\beta_\m \bdt^\m  \in \Image\psi$, set
\begin{eqnarray}
X(a,\n, f(\bdt))=\sum\limits_{\ranger{m}}\beta_\m a_{-1,\m+\n}(1\otimes \bdt^{-\m})\in V(\psi,0).
\end{eqnarray}
 Now, we are ready to give a characterization of $L(\psi,0)$.

\begin{prop}\label{GenProp}
Set $\ell=\ell_1+\dots+\ell_s$. Then
$$L(\psi,0)=V(\psi,0)\big/J,$$
where $J$ is the ideal of $(r+1)$-toroidal vertex algebra $V(\psi,0)$, generated by
\begin{eqnarray}
\left(e_\theta(-1,\m)\right)^{\ell+1}\vac,\   \    \   X(a,\n,f(\bdt))
\end{eqnarray}
for $\ranger{m},\ a\in \g,~\ranger{n},~ f(\bdt)\in  p_1(t_1)(\Image\psi)+\cdots +p_r(t_r)(\Image\psi)\subset L_r$.
\end{prop}

\begin{proof} For $a\in \g$, $\ranger{n}$ and $f(\bdt)=\sum\limits_{\ranger{m}}\beta_\m \bdt^\m  \in \Image\psi$,
we have
$$\gamma (X(a,\n,f(\bdt)))=\sum_{\m\in \Z^r}\beta_{\m}(a\otimes t_0^{-1}\bdt^{\m+\n})\vac
=(a\otimes t_0^{-1}\bdt^{\n}f(\bdt))\vac.$$
Note that as $\Image\psi$ is a nonzero $\Z^r$-graded subring of $L_r$, we have $\big(\Image\psi\big)L_r=L_r$. Then
$\bdt^{\n}p_j(t_j)\big(\Image\psi\big)$ for $\n\in \Z^r$ linearly span $p_j(t_j)L_r$.
Since the generators of $J$ are homogeneous, $J$ is simply the $\tau$-submodule generated by those generators.
It then follows that $\gamma(J)=I$.

Recall the commutative diagrams (\ref{ecomm-diagram1}) and (\ref{ecomm-diagram2}).
As $L(\psi,0)$ is graded simple, by Lemma \ref{LrLem},  $\tilde{\pi}$ is injective.
By Lemma \ref{lem:tildegammaProp}, $\tilde{\gamma}$ is also injective.
Then
 $$\ker \alpha=\tilde{\gamma}^{-1}\left(\ker (\alpha^{0}\otimes 1)
 \right)=\tilde{\gamma}^{-1}\left(\lr{I} \right).$$ As $\gamma(J)=I$, it can be
 readily seen that the ideal of vertex algebra $\lr{\Vzero{\psi}}$ generated by $\tilde{\gamma}\left(J\right)$
 is equal to $\lr{I}$.
By Lemma \ref{lem:IdealCorrespLem222222}, we have
$\tilde{\gamma}(J)=\Image\tilde{\gamma}\cap \lr{I}$,
which is equivalent to  $J=\tilde{\gamma}^{-1}\left(\lr{I}\right)$.
Thus, $J=\tilde{\gamma}^{-1}\left(\lr{I}\right)=\ker \alpha$.
Therefore, $V(\psi,0)\big/J=L(\psi,0)$.
\end{proof}

Recall that a $V(\psi,0)$-module structure on a vector space amounts to a restricted $\tau$-module
structure with $\ker \psi$ acting trivially.
We have:

\begin{thm}\label{ModCorrespThm10086}
 An $\Stva{\psi}$-module structure on a vector space $W$ exactly amounts to
 a restricted and integrable $\Toro$-module structure such that
 $(\ker \psi) W=0$ and
\begin{eqnarray}\label{ModCorrespThm10086Eq10001}
  p_j(x_j/t_j)\eltAbs{a}{x_0}{\x}=0\   \mbox{ on }\ W\   \mbox{ for all }1\le j\le r,\ a\in \g.
\end{eqnarray}
\end{thm}

\begin{proof}
Assume that $W$ is an $\Stva{\psi}$-module. We know that $W$ is a restricted $\Toro$-module.
From the definition of $\Stva{\psi}$, we get that $(\ker \psi)W=0$.
By Theorem \ref{IntegrableProp123}, $L(\psi,0)$ is an integrable $\tau$-module.
Then for $\alpha\in \Delta$, $e_\alpha\in \g_\alpha$ and $\m\in\Zr$, there exists a positive integer $k$ such that
$e_\alpha(-1,\m)^k\vac=0$. Thus we have
\begin{eqnarray}
e_\alpha(x_0,\n)^{k}=Y^{0}_{W}(e_{\alpha}(-1,\n)^{k}{\bf 1},x_{0})=0,
\end{eqnarray}
noticing that $[e_\alpha(m_0,\n),e_\alpha(n_0,\n)]=0$ for $m_0,n_0\in \Z$.
Let $w\in W$.
As $W$ is a restricted $\tau$-module, there exists $q\in \Z$ such that
$e_\alpha(n_0,\n)w=0$ for $n_0> q$. Suppose that $p$ is an integer such that
$e_\alpha(n_0,\n)$ are nilpotent on $w$ for $n_0\ge p$.
Then $e_\alpha(x_0,\n)-\sum\limits_{n_0=p}^{	q}e_\alpha(n_0,\n)x_0^{-n_0-1}$ is nilpotent on $w$.
It follows that $e_\alpha(p-1,\n)$ is nilpotent on $w$.
By induction on $p$, we get that $e_\alpha(n_0,\n)$ are nilpotent on $w$ for $n_0\in \Z$.
This shows that $W$ is integrable.
Note that
\begin{eqnarray}\label{eq:temptwoEquations2048}
&&\sum_{\n\in \Z^r}Y^0_{W}(X(a,\n,f(\bdt)){\bf 1},x_{0})\x^{-\n}=f(\x/\bdt)a(x_0,\x)
\end{eqnarray}
for $f(\bdt)\in p_1(t_1)({\rm Im}\psi)+\cdots+p_r(t_r)({\rm Im}\psi)$.
Then from Proposition \ref{GenProp} we see that \eqref{ModCorrespThm10086Eq10001} holds.

For the other direction, assume $W$ is a restricted and integrable $\Toro$-module satisfying the very condition.
First, $W$ is naturally a module for $V(\Sc,0)$.
Since $(\ker \psi)W=0$, $W$ is a module for $V(\psi,0)$.
In view of Proposition \ref{GenProp}, we only need to show that for each  $\ranger{n}$,
\begin{eqnarray}\label{ethetal+1}
e_\theta(x_0,\n)^{\ell+1}=0\   \mbox{ on } W,
\end{eqnarray}
where $\ell=\ell_1+\dots +\ell_s$.
For $\ranger{n}$, set
\[
\s_\n=\g_\theta\otimes \bdt^\n\C\left[t_0^{\pm 1}\right]+ \g_{-\theta}\otimes \bdt^{-\n}\C\left[t_0^{\pm 1}\right]  +
	\C\check{\theta}\otimes \C\left[t_0^{\pm 1}\right] + \C\cent.
\]
Notice that $\s_\n$ is a subalgebra of $\lr{\hat{\g}}$, which is isomorphic to $\widehat{\mathfrak{sl}_2}$.
Since $\psi(\ell-\cent)=\ell-\sum_{i=1}^{s}\ell_i=0$ and $(\ker \psi)W=0$,  $\cent$ acts on $W$ as scalar  $\ell$.
As a restricted and integrable $\s_\n$-module, from \cite{DLM} $W$ is a direct sum of irreducible highest weight integrable modules
of level $\ell$.
Then   (see \cite{LP}) we have (\ref{ethetal+1}).
Therefore, $W$ is an $\Stva{\psi}$-module.
\end{proof}

Recall that  $\tilde{\gamma}$ is an \rtva{} homomorphism from
$\T{\psi}$ to $\lr{\Vzero{\psi}}$.  Through $\tilde{\gamma}$, every
$\lr{\Vzero{\psi}}$-module is naturally a $\T{\psi}$-module. On the other hand, we have:

\begin{lem}\label{lem:ModLiftingLem}
Let $(W,Y_W)$ be a $\T{\psi}$-module. Suppose that $(W,\rho)$ is also an
$\Lr$-module such that
\begin{eqnarray*}
Y_W(\bdt^\m;x_0,\x)=\rho(\bdt^\m)&\andtext
&Y_W(v;x_0,\x)\rho(\bdt^\n)=\rho(\bdt^\n)Y_W(v;x_0,\x)
\end{eqnarray*}
for $v\in \T{\psi}$, $\m\in\Lambda(\psi)$ and $\n\in\Zr$.
Then there exists an $\lr{\Vzero{\psi}}$-module structure $Y_W'$
on $W$ such that
\begin{eqnarray*}
Y_W'(\tilde{\gamma}(v);x_0,\x)=Y_W(v;x_0,\x) \  \   \mbox{ for } v\in \T{\psi}.
\end{eqnarray*}
\end{lem}

\begin{proof} Recall from Proposition \ref{prop:ZrGradedVAtoETVA} that
$\T{\psi}$ and $\lr{\Vzero{\psi}}$ are naturally $\Zr$-graded vertex algebras. Denote by
$Y^{(1)}$ and $Y^{(2)}$ the vertex operator maps of $\T{\psi}$ and
$\lr{\Vzero{\psi}}$, respectively.
From Theorem \ref{ZeroModCateProp},  $W$ is a module for $\T{\psi}$ viewed as a vertex algebra.
Denote by $Y_W^{(1)}$ the $\T{\psi}$-module structure on $W$.
Recall that
\begin{eqnarray*}
\T{\psi}= U(L_r(\g\otimes t_0^{-1}\C[t_0^{-1}]))\otimes
\Image\psi\  \mbox{ and }\
\Vzero{\psi}= U(L_r(\g\otimes t_0^{-1}\C[t_0^{-1}]))
\end{eqnarray*}
as vector spaces.
Then $\Vzero{\psi}$ is a $\Zr$-graded vector space.
For $\m\in \Z^r$, let $U(\m)$ denote the homogeneous subspace of $\Vzero{\psi}$ of degree $\m$.
Define a linear map $Y_W^{(2)}(\cdot,x_0):\lr{\Vzero{\psi}}\rightarrow \Hom (W,W((x_0)))$ by
\begin{eqnarray*}
Y_W^{(2)}(u\otimes \bdt^\n,x_0)=\rho(\bdt^{\n-\m})Y_W^{(1)}(u,x_0)\ \ \ \ \mbox{for }u\in U(\m),\ \m, \n\in\Zr.
\end{eqnarray*}
From definition, for $u\in U(\m),\ \n\in \Lambda(\psi)$, we have
\begin{eqnarray*}
Y_W^{(2)}(\tilde{\gamma}(u\otimes \bdt^\n),x_0)=Y_W^{(2)}(u\otimes \bdt^{\m+\n},x_0)
=\rho(\bdt^\n)Y_W^{(1)}(u,x_0)=Y_W^{(1)}(u\otimes\bdt^\n,x_0).
\end{eqnarray*}
That is,
\begin{eqnarray*}
Y_W^{(2)}(\tilde{\gamma}(v),x_0)=Y_W^{(1)}(v,x_0)\ \ \ \ \mbox{for }v\in \T{\psi}.
\end{eqnarray*}

We now show that $Y_W^{(2)}$ satisfies the Jacobi identity.
Let $u\in U(\m)$, $v\in U(\n),\ \m,\n,\ \rsymbol{p},\rsymbol{q}\in \Zr$. We have
\begin{eqnarray*}
&&\dfunc{z}{x-y}{z}Y_W^{(2)}(u\otimes \bdt^{\m+\rsymbol{p}},x)Y_W^{(2)}(v\otimes\bdt^{\n+\rsymbol{q}},y)\\
&&\quad\quad	-\dfunc{z}{y-x}{-z}Y_W^{(2)}(v\otimes\bdt^{\n+\rsymbol{q}},y)Y_W^{(2)}(u\otimes \bdt^{\m+\rsymbol{p}},x)\\
&=&\dfunc{z}{x-y}{z}\rho(\bdt^{\rsymbol{p}})Y_W^{(1)}(u,x)\rho(\bdt^{\rsymbol{q}})Y_W^{(1)}(v,y)\\
&&\quad\quad	-\dfunc{z}{y-x}{-z}\rho(\bdt^{\rsymbol{q}})Y_W^{(1)}(v,y)\rho(\bdt^{\rsymbol{p}})Y_W^{(1)}(u,x)\\
&=&\dfunc{x}{y+z}{x}\rho(\bdt^{\rsymbol{p}+\rsymbol{q}})Y_W^{(1)}(Y^{(1)}(u,z)v,y).
\end{eqnarray*}
Note that
\begin{eqnarray*}
&&Y_W^{(1)}(Y^{(1)}(u,z)v,y)=Y_W^{(2)}\left(\tilde{\gamma}\left( Y^{(1)}(u,z)v \right),y\right)
=Y_W^{(2)}\left( Y^{(2)}(\tilde{\gamma}(u),z)\tilde{\gamma}(v ),y\right)\\
&&\quad=Y_W^{(2)}\left( Y^{(2)}(u\otimes\bdt^\m,z)(v\otimes\bdt^\n) ,y\right)
=Y_W^{(2)}\left( Y^{(2)}(u,z)v\otimes\bdt^{\m+\n} ,y\right)
\end{eqnarray*}
and
\begin{eqnarray*}
&&\rho(\bdt^\n)Y_W^{(2)}(u\otimes \bdt^{\rsymbol{p}},x)
=\rho(\bdt^{\n+\rsymbol{p}-\m})Y_W^{(1)}(u,x)=Y_W^{(2)}(u\otimes\bdt^{\n+\rsymbol{p}},x).
\end{eqnarray*}
Then
\begin{eqnarray*}
&&\rho(\bdt^{\rsymbol{p}+\rsymbol{q}})Y_W^{(1)}(Y^{(1)}(u,z)v,y)
=\rho(\bdt^{\rsymbol{p}+\rsymbol{q}})Y_W^{(2)}\left( Y^{(2)}(u,z)v\otimes\bdt^{\m+\n} ,y\right)\\
&&\quad=Y_W^{(2)}\left( Y^{(2)}(u,z)v\otimes\bdt^{\m+\rsymbol{p}+\n+\rsymbol{q}} ,y\right)
=Y_W^{(2)}\left( Y^{(2)}(u\otimes\bdt^{\m+\rsymbol{p}},z)(v\otimes\bdt^{\n+\rsymbol{q}}) ,y\right).
\end{eqnarray*}
Now the Jacobi identity follows. Thus, $(W,Y_W^{(2)})$ is an $\lr{\Vzero{\psi}}$-module.
By Theorem \ref{ZeroModCateProp}, $Y_W^{(2)}$ gives rises to a module structure $Y_W'$ on $W$
for $\lr{\Vzero{\psi}}$ viewed as an \rtva{}, as desired.
\end{proof}

Furthermore, we have:

\begin{prop}\label{prop:SimpleModLiftingProp}
Let $(W,Y_W)$ be any irreducible $\T{\psi}$-module.
Then there exists an (irreducible)  $\lr{\Vzero{\psi}}$-module structure $Y_W'$ on $W$ such that
\begin{eqnarray}\label{eext-rest}
Y_W'(\tilde{\gamma}(v);x_0,\x)=Y_W(v;x_0,\x) \  \   \mbox{ for } v\in \T{\psi}.
\end{eqnarray}
\end{prop}

\begin{proof} Recall that any $\T{\psi}$-module is naturally an
$\Image\psi$-module. Denote by $\rho$ the representation map of $\Image\psi$ on $W$.
In view of Lemma \ref{lem:ModLiftingLem}, it suffices to prove that there exists a representation  $\rho'$ of $\Lr$ on $W$ such that
\begin{eqnarray*}
&&\rho'(\bdt^\m)=\rho(\bdt^\m) \  \   \mbox{ for } \m\in\Lambda(\psi),\\
&&Y_W(v;x_0,\x)\rho'(\bdt^\n)=\rho'(\bdt^\n)Y_W(v;x_0,\x)
\end{eqnarray*}
for $v\in \T{\psi}$ and $\n\in\Zr$.
As an irreducible $\T{\psi}$-module, $W$ is naturally an irreducible
$\tau$-module. Since $\tau$ is of countable dimension (over $\C$), it follows that $W$ is of countable dimension.
 As $\Sc$ lies in the center of $U(\tau)$, each element of $\Sc$ acts on $W$ as a scalar.
Consequently, each element of $\Image\psi$ acts on $W$ as a scalar.
Since $\Lambda(\psi)$ is a cofinite subgroup of $\Zr$ (see Corollary \ref{coro:ImPsiGp}),
there exist $\m_1,\dots,\m_r\in \Zr$ and positive integers $d_1,\dots,d_r$ such that
\begin{eqnarray*}
&&\Zr=\Z\m_1\oplus \cdots\oplus \Z\m_r,\\
&&\Lambda(\psi)=d_1\Z\m_1\oplus \cdots \oplus d_r\Z\m_r.
\end{eqnarray*}
For $1\leq i\leq r$, assume that $\bdt^{d_i\m_i}$ acts as scalar $c_i$ on $W$ and
choose a $d_i$-th root $(c_i)^{\frac{1}{d_i}}$ of $c_i$.
We then define an $\Lr$-module structure on $W$ by
\begin{eqnarray*}
\bdt^{\m_i}\cdot w=(c_i)^{\frac{1}{d_i}}w \   \   \   \mbox{ for }1\leq i\leq r,\ w\in W.
\end{eqnarray*}
This action of $L_r$ clearly extends that of $\Image\psi$. The other assertion is also clear.
\end{proof}

Recall that $\tilde{\pi}$ is an \rtva{} homomorphism from $\Stva{\psi}$ to $\lr{\VSzero{\psi}}$.
We have the following analogue of Proposition \ref{prop:SimpleModLiftingProp}:

\begin{thm}\label{SimpleModClassifyThm10001}
Let $(W,Y_W)$ be any irreducible $\Stva{\psi}$-module.
Then there exists an irreducible $\lr{\VSzero{\psi}}$-module structure $Y_W''$ on $W$ such that
\begin{eqnarray*}
Y_W''(\tilde{\pi}(u);x_0,\x)=Y_W(u;x_0,\x) \  \   \mbox{ for } u\in \Stva{\psi}.
\end{eqnarray*}
\end{thm}

\begin{proof} Recall that $\VSzero{\psi}=\Vzero{\psi}/I$ and $\Stva{\psi}=\T{\psi}/J$.
Note that $W$ is naturally an irreducible $V(\psi,0)$-module.
By Proposition \ref{prop:SimpleModLiftingProp}, there exists an $\lr{\Vzero{\psi}}$-module structure
 $Y_W'$ on $W$ satisfying (\ref{eext-rest}).
From the proof of Proposition \ref{GenProp}, we see that $\lr{I}$ is generated by $\tilde{\gamma}(J)$.
It follows that $Y_{W}'$ gives rise to an $\lr{\VSzero{\psi}}$-module structure $Y_W''$ on $W$ such that
\begin{eqnarray*}
Y_W''(\tilde{\pi}(v);x_0,\x)=Y_W(v;x_0,\x)\ \ \ \mbox{ for } v\in\Stva{\psi}.
\end{eqnarray*}
Since $W$ is an irreducible $\Stva{\psi}$-module, $(W,Y_W'')$ must be also
irreducible.
\end{proof}

Note that Theorem \ref{SimpleModClassifyThm10001} gives
a classification of irreducible $\Stva{\psi}$-modules
 in terms of irreducible $\lr{\VSzero{\psi}}$-modules, while
 Propositions \ref{IrrTensorProp} and \ref{SzeroStructProp}  give a classification of
 irreducible $\lr{\VSzero{\psi}}$-modules in terms of integrable highest weight $\hat{\g}$-modules.
On the other hand, using Theorems \ref{ModCorrespThm10086} and \ref{SimpleModClassifyThm10001}
we immediately have:

\begin{coro}
Let $W$ be an irreducible restricted and integrable $\Toro$-module satisfying the conditions that
 $(\ker \psi)W=0$ and that
 for all  $1\le j\le r,\  a\in \g$,
 $$ p_j(x_j/t_j)\eltAbs{a}{x_0}{\x}=0\   \mbox{ on }  W.$$
 Then $W$ is isomorphic to a $\tau$-module of the form
$\left(W_1\otimes\cdots\otimes W_s\right)_{\rsymbol{c}}$,
 where $W_i$ are integrable highest weight $\hat{\g}$-modules
of level $\ell_i$ $(1\le i\le s)$,  $\rsymbol{c}\in (\C^\times)^r$, and
\begin{eqnarray*}
&&(u\otimes \bdt^\m)\cdot(w_1\otimes \cdots \otimes w_s)=\rsymbol{c}^\m\sum_{i=1}^s  \bda_i^\m
	\left(w_1\otimes \cdots \otimes u w_i\otimes \cdots \otimes w_s\right)
\end{eqnarray*}
for $u\in \hat{\g}$, $\m\in \Z^r$, $w_i\in W_i$ with $1\leq i\leq s$.
\end{coro}


\begin{thebibliography}{AAGBP}

\bibitem[AABGP]{AABGP}
B. N. Allison, S. Azam, S. Berman, Y. Gao, and A. Pianzola, Extended
affine Lie algebras and their root systems, {\em Memoirs Amer. Math.
Soc.} {\bf 126}, 1997.

\bibitem [BBS]{BBS}
S. Berman, Y. Billig, and J. Szmigielski, Vertex operator algebras
and the representation theory of toroidal algebras, Contemporary
Math. {\bf 297}, Amer. Math. Soc., Providence, 2002, 1-26.

\bibitem[DL]{DL}
C. Dong and J. Lepowsky, \emph{Generalized Vertex Algebras and Relative Vertex Operators},
Progress in Math. Vol 1. 112, Birkh{\"{a}}user, Boston, 1993.

\bibitem[DLM]{DLM}
C. Dong, H. Li, and G. Mason, Regularity of rational vertex operator algebras,
{\em Advances in Math.} {\bf 132} (1997), 148-166.

\bibitem[DM]{DM}
C. Dong and G. Mason, On quantum Galois theory, {\em Duke Math. J.} {\bf 86} (1997), 305-321.

\bibitem[FHL]{FHL}
I. B. Frenkel, Y. Huang, and J. Lepowsky,
{\em On Axiomatic Approaches to Vertex Operator Algebras and Modules},
Amer. Math. Soc. {\bf 104} (1993) no. 494.

\bibitem[FLM]{FLM}
I. B. Frenkel, J. Lepowsky, and A. Meurman, {\em Vertex Operator
Algebras and the Monster}, Pure and Applied Math. Vol. 134,
Academic Press, Boston, 1988.

\bibitem[FZ]{FZ}
I. B. Frenkel and Y. Zhu, Vertex operator algebras associated to representations of affine and
Virasoro algebras, {\em Duke Math. J.} {\bf 66} (1992), 123-168.

\bibitem[GTW]{GTW}
H. Guo, S. Tan, and Qing Wang, Some categories of modules for toroidal Lie algebras,
{\em J. Algebra} {\bf 401} (2014), 125-143.

\bibitem[IKU]{IKU}
T. Inami, H. Kanno, and T. Ueno, Higher-dimensional WZW
Model on K\"{a}hler Manifold and Toroidal Lie Algebra, {\em Mod. Phys.
Lett.} {\bf A12} (1997), 2757-2764.

\bibitem[IKUX]{IKUX}
T. Inami, H. Kanno, T. Ueno, and C.-S Xiong,
Two-toroidal Lie Algebra as Current Algebra of Four-dimensional
K\"{a}hler WZW Model, {\em Phys. Lett.} {\bf B399} (1997), 97-104.

\bibitem[K]{K}
V. G. Kac, \emph{Infinite-dimensional Lie Algebras}, 3rd ed.,
Cambridge University Press, Cambridge, 1990.

\bibitem[LL]{LL}
J. Lepowsky and H. Li, \emph{Introduction to Vertex Operator Algebras and Their Representations},
Progress in Math. Vol. {\bf 227}, Birkh{\"{a}}user, Boston, 2004.

\bibitem[LP]{LP}
J. Lepowsky and M. Primc, Structure of the standard modules for the affine Lie algebra $A_1^{(1)}$,
{\em Contemporary Mathematics} {\bf 46} (1985).

\bibitem[Li1]{Li1}
H. Li, Local systems of vertex operators, vertex superalgebras and modules, \emph{J. Pure Applied Algebra}
{\bf 109} (1996), 143-195.



\bibitem[Li2]{Li4}
H. Li, Representation theory and tensor product theory for modules for a vertex operator algebra, Ph.D. thesis, Rutgers University, 1994.

\bibitem[LTW1] {LTW}
H. Li, S. Tan, and Q. Wang, Toroidal vertex algebras and their modules,
{\em J. Algebra} {\bf 365} (2012), 50-82.

\bibitem[LTW2] {LTW2}
H. Li, S. Tan, and Q. Wang, On vertex Leibniz algebras,
{\em J. Pure. Appl. Algebra} {\bf 217} (2013), 2356-2370.

\bibitem[MP1]{MP1}
A. Meurman and M. Primc, Vertex Operator Algebras and Representations of Affine Lie Algebras,
{\em Acta Applicandae Math.} {\bf 44} (1996), 207-215.

\bibitem[MP2]{MP2}
A. Meurman and M. Primc, Annihilating Fields of Standard Modules of $\widetilde{\mathfrak{sl}(2,\C)}$ and Combinatorial Identities,
{\em Memoirs Amer. Math. Soc.} {\bf 652}, 1999.

\bibitem[R1]{R1}
S. Eswara Rao, Classification of irreducible integrable modules for multi-loop algebras with finite
dimensional weight spaces, {\em J. Algebra} {\bf 246} (2001), 215-225.

\bibitem[R2]{R2}
S. Eswara Rao, Classification of irreducible integrable modules for toroidal Lie algebras with finite dimensional weight spaces,
{\em J. Algebra} {\bf 277} (2004), 318-348.

\end{thebibliography}
\end{document}